\documentclass[10pt]{amsart}
\usepackage[style=alphabetic]{biblatex}
\usepackage{tikz}
\usetikzlibrary{arrows.meta,positioning,calc}
\usepackage{amssymb}
\usepackage{amsthm}
\usepackage[all]{xy}
\usepackage{xcolor}
\usepackage{thmtools}
\usepackage{hyperref}
\usepackage[nameinlink]{cleveref}
\usepackage{graphicx}
\usepackage{mathrsfs}
\usepackage{tikz-3dplot}
\usepackage{mathtools}
\usepackage{enumitem}
\usepackage[section]{placeins}
\usepackage{microtype}

\makeatletter

\@addtoreset{equation}{section}
\makeatother

\hypersetup{
 colorlinks,
 linkcolor={teal},
 citecolor={teal},
 urlcolor={teal},
}

\calclayout

\DeclareFieldFormat
  [article,book,inbook,incollection,inproceedings,patent,thesis,unpublished]
  {title}{\emph{#1\isdot}}

\theoremstyle{plain}
\newtheorem{theorem}{Theorem}[section]
\newtheorem{lemma}[theorem]{Lemma}
\newtheorem{corollary}[theorem]{Corollary}
\newtheorem{proposition}[theorem]{Proposition}
\newtheorem{conjecture}[theorem]{Conjecture}

\theoremstyle{definition}
\newtheorem{definition}[theorem]{Definition}
\newtheorem{example}[theorem]{Example}
\newtheorem{remark}[theorem]{Remark}



\newcommand{\zero}{\mathbf{0}}
\newcommand{\one}{\mathbf{1}}
\newcommand{\NN}{\mathbb{N}}
\newcommand{\R}{\mathbb{R}}
\newcommand{\ZZ}{\mathbb{Z}}

\DeclareMathOperator{\sgn}{sgn}
\DeclareMathOperator{\conv}{conv}
\newcommand{\A}{\mathcal{A}}
\newcommand{\B}{\mathcal{B}}
\newcommand{\C}{\mathcal{C}}
\newcommand{\D}{\mathcal{D}}

\newcommand{\T}{\mathcal{T}}
\newcommand{\F}{\mathcal{F}}
\newcommand{\G}{\mathcal{G}}

\DeclareMathOperator{\id}{id}
\newcommand{\MC}{\mathrm{MC}}
\newcommand{\MH}{\mathrm{MH}}

\newcommand{\Bl}{\mathrm{Bl}}

\newcommand{\M}{\mathcal{M}}
\newcommand{\N}{\mathcal{N}}
\newcommand{\X}{\mathsf{X}}
\newcommand{\Lcov}{\mathcal{L}}
\newcommand{\Sep}{\operatorname{Sep}}
\newcommand{\rk}{\operatorname{rk}}
\newcommand{\cl}{\operatorname{cl}}

\newcommand{\Hom}{\operatorname{Hom}}

\newcommand{\eps}{\varepsilon}

\newcommand{\Ftwo}{\mathbb F_2}
\newcommand{\kk}{\Bbbk}
\newcommand{\bb}{\mathfrak{b}}
\newcommand{\cc}{\mathfrak{c}}
\newcommand{\bS}{\boldsymbol{S}}
\newcommand{\bA}{\boldsymbol{A}}
\newcommand{\bD}{\boldsymbol{D}}
\newcommand{\bF}{\boldsymbol{F}}
\newcommand{\bH}{\boldsymbol{H}}
\newcommand{\bG}{\boldsymbol{G}}

\begin{document}

\title{Mod 2 magnitude cohomology ring of real hyperplane arrangements}
\author[Y. Liu]{Ye Liu}
\address{Department of Pure Mathematics, Xi’an Jiaotong-Liverpool University, Suzhou, Jiangsu 215123, P. R. China}
\email{yeliumath@gmail.com}

\date{\today}

\begin{abstract}
Let $\mathcal{A}$ be a finite central real hyperplane arrangement and let $\mathcal{G}(\mathcal{A})$ be its tope graph. Koizumi recently proved that the crossing-graded magnitude homology of $\mathcal{G}(\mathcal{A})$ is torsion-free and is freely indexed by face flags. We refine his result by proving that a fixed crossing vector and terminal chamber determine a summand of rank at most one. Over $\mathbb{F}_2$, we use the canonical cohomology basis to determine the crossing-graded magnitude cohomology ring of $\mathcal{G}(\mathcal{A})$.
\end{abstract}

\maketitle
\setcounter{tocdepth}{1}
\tableofcontents

\section{Introduction}

Magnitude is a cardinality-like invariant of enriched categories and metric spaces introduced by Leinster \cite{Leinster2013,Leinster2019}. Hepworth and Willerton constructed magnitude homology for graphs and showed that its graded Euler characteristic recovers magnitude \cite{Hepworth2017}. Hepworth subsequently introduced magnitude cohomology and equipped it with an associative, generally noncommutative cup product \cite{Hepworth2022}. The resulting ring is genuinely stronger than the underlying groups. In particular, for a finite metric space, its magnitude cohomology ring recovers the metric.

Real hyperplane arrangements provide a particularly structured family of finite metric spaces. Let $\A=\{H_e\}_{e\in E}$ be a finite real hyperplane arrangement, a connected component of the complement is called a chamber. The \emph{tope graph} $\G(\A)$ has vertices the chambers, and two chambers are adjacent if they are separated by exactly one hyperplane. It is a finite metric space with the usual graph path metric. The tope graph $\G(\A)$ is known to carry sufficient information to recover the underlying oriented matroid up to relabeling and reorientation (\cite[Theorem 4.2.14]{Bjorner1999}), and therefore it determines the topology of the complexified complement and various combinatorial and algebraic invaraints of the arrangement (see \cite{Yagi2026}).

Recently, the study of magnitude theory of tope graphs $\G(\A)$ has been initialized in \cite{Koizumi2026}. Using special features of tope graphs, Koizumi and Liu managed to prove that the magnitude and magnitude homology of $\G(\A)$ for central arrangements $\A$ possess many nice structural properties. Later, Koizumi \cite{Koizumi2026a} completely determined the integral magnitude homology of $\G(\A)$ for affine arrangements $\A$. A key refinement made by Koizumi is that the magnitude length parameter can be refined by a \emph{crossing vector}. Among others, he proved that this homology is torsion free and is indexed by a combinatorial structure called face flags. In particular, it is determined by the intersection poset of the arrangement.

The aim of this paper is to study the magnitude cohomology ring of $\G(\A)$ for central arrangements $\A$. We first refine Koizumi's face flag theorem by noticing that for a terminal chamber $B$ and a prescribed crossing vector $\alpha$, the indexed magnitude homology has rank at most one. The nontrivial case has a combinatorial characterization that $(\alpha,B)$ is \emph{tope-admissible}. Then taking dual over $\Ftwo$, each tope admissible pair $(\alpha,B)$ indexes a canonical nontrivial cohomology class $u_{\alpha,B}$. The novelty of this paper is that we give a complete multiplicative table of these classes (\Cref{product-formula}). 

Let us outline the proof. For a crossing vector $\alpha\in\NN^E$, its upper level sets $F_j:=U_j(\alpha)=\{e\in E\mid \alpha_e\geq j\}$ form a decreasing sequence of subsets of the ground set $E$. Under the assumption of tope-admissibility, these $F_j$ are flats of the underlying matroid. Moreover, starting from the terminal chamber $B_m=B$, reflecting across the flats in the order $F_m,\ldots,F_1$ produces a sequence of chambers $B_{m-1},\ldots,B_0$ such that $(\one_{F_i},B_i)$ is tope-admissible and a standard factorization
\[
u_{\alpha,B}=u_{\one_{F_1},B_1}\smile\cdots\smile u_{\one_{F_m},B_m}.
\]
This reduces the problem to single-flat classes and a sorting algorithm. The product formula is quite natural. For $(\alpha,B)$ and $(\beta,C)$ tope-admissible, we show that the product $u_{\alpha,B}\smile u_{\beta,C}$ is trivial, unless
\begin{itemize}
    \item their standard factorizations can be concatenated, and
    \item $(\alpha+\beta,C)$ is tope-admissible, and
    \item the cohomology degree of $u_{\alpha,B}\smile u_{\beta,C}$ agrees with that of $u_{\alpha+\beta,C}$.
\end{itemize}
The third condition turns out to be equivalent to that there is a basis of the underlying matroid which simultaneously maximizes $\alpha$ and $\beta$, which naturally brings weighted matroid optimization theory to our scope. When these three conditions are satisfied, we have
\[
u_{\alpha,B}\smile u_{\beta,C}=u_{\alpha+\beta,C}.
\]

The structure of the paper is as follows. In \Cref{prelim}, we review necessary materials of (oriented) matroids and magnitude theory. \Cref{matroid-optimization} discusses matroid optimization theory and prove some technical results used later. \Cref{periodicty-section} refines Koizumi's face flag theorem and define mod $2$ magnitude cohomology classes. In \Cref{integral-basis}, we construct an integral basis for magnitude homology with antipodal endpoints, which is used in \Cref{local-modular} for proving the local modular relation. In \Cref{main-result}, we assemble all previous results to prove the product formula.

\subsection*{Glossary} Throughout this paper, we use the following letters for specific objects.
\begin{itemize}
    \item $\NN$: the set of nonnegative integers.
    \item $E$: a finite set (the ground set).
    \item $e,f,g$: elements of $E$.
    \item $S,T,U$: (general) subsets of $E$.
    \item $\alpha,\beta,\gamma$: functions/vectors in $\NN^E$.
    \item $\one_S$: the characteristic vector of $S\subseteq E$.
    \item $\zero$: the zero vector/function.
    \item $V,W$: finite dimensional real vector spaces.
    \item $\A$: a central real hyperplane arrangement in $V$ indexed by $E$.
    \item $\M$: an oriented matroid on $E$ (underlying oriented matroid of $\A$).
    \item $M$: a matroid on $E$ (underlying matroid of $\M$).
    \item $A,B,C,D$: chambers/topes of $\A$ (or $\M$).
    \item $X,Y,Z$: faces/covectors of $\A$ (or $\M$).
    \item $F,G$: flats of $M$.
\end{itemize}

\section{Preliminaries}\label{prelim}

\subsection{Matroids and oriented matroids}\label{matroid-background}
Standard references of matroids and oriented matroids are \cite{Oxley2011,Bjorner1999,Anderson2025}. For references of hyperplane arrangements, see \cite{Stanley2007,Orlik1992}.

\subsubsection{Matroids}
Let $M$ be a matroid on the finite ground set $E$, with rank function $\rk_M$. We abbreviate $\rk_M(S)$ to $\rk S$ when the ambient matroid is clear. The rank of $M$ is defined by $\rk M=\rk E$. The closure of $S\subseteq E$ is
\[
\cl_M(S)=\{e\in E \mid \rk(S\cup\{e\})=\rk S\},
\]
and a subset $F\subseteq E$ is a \emph{flat} if $F=\cl_M(F)$. The flats form a geometric lattice $L(M)$ where the meet is intersection and the join is $F\vee G=\cl_M(F\cup G)$. We call a pair of flats $(F,G)$ \emph{modular} if
\[
\rk F+\rk G=\rk(F\cap G)+\rk(F\vee G).
\]

A subset $I\subseteq E$ is \emph{independent} if $\rk I=|I|$, and a \emph{basis} is a maximal independent subset. Let $\B(M)$ be the set of bases of $M$. A \emph{circuit} is a dependent subset that is minimal under inclusion. Thus every proper subset of a circuit is independent. An element is a \emph{loop} if it belongs to no basis, and a \emph{coloop} if it belongs to every basis. Two distinct nonloops $e,f$ are \emph{parallel} if they form a circuit, equivalently if $\rk\{e,f\}=1$. A matroid is \emph{simple} if it has no loops and no parallel elements.

If $\bb\in\B(M)$ is a basis and $e\not\in\bb$, then $\bb\cup\{e\}$ contains a unique circuit, called the \emph{fundamental circuit} of $e$ with respect to $\bb$ and is denoted by $\C_M(e,\bb)$. It records the possible basis exchanges,
\[
f\in \C_M(e,\bb)\setminus\{e\}\iff(\bb\setminus\{f\})\cup\{e\}\in\B(M).
\]

The following useful lemma is an immediate consequence of \cite[Proposition~1.1.6, Corollary~1.2.6, and Proposition~1.4.11]{Oxley2011}.

\begin{lemma}\label{fund-circuit-lemma}
Let $\bb\in\B(M)$ and $e\in E\setminus\bb$. For every subset $I\subseteq\bb$, the following are equivalent:
\begin{enumerate}[label=\textup{(\roman*)}]
\item $e\in\cl_M(I)$;
\item $\mathcal C_M(e,\bb)\subseteq I\cup\{e\}$.
\end{enumerate}
\end{lemma}

For $S\subseteq E$, the \emph{restriction} $M|S$ is the matroid on $S$ with
\[
\rk_{M|S}(T)=\rk_M(T),\quad T\subseteq S.
\]
The \emph{contraction} $M/S$ is the matroid on $E\setminus S$ with
\[
\rk_{M/S}(T)=\rk_M(T\cup S)-\rk_M(S),\quad T\subseteq E\setminus S.
\]
If $M_1$ and $M_2$ have disjoint ground sets, their direct sum $M_1\oplus M_2$ is characterized by
\[
\rk_{M_1\oplus M_2}(T_1\sqcup T_2)=\rk_{M_1}(T_1)+\rk_{M_2}(T_2).
\]

For a matroid $M$, two elements $e,f\in E$ lie in the same \emph{connected component} of $M$ if $e=f$ or some circuit contains both $e$ and $f$. This is an equivalence relation. If its equivalence classes are $E_1,\ldots,E_t$, then
\[
M=(M|E_1)\oplus\cdots\oplus(M|E_t),
\]
and none of the factors admits a further nontrivial direct sum decomposition.

A subset $S\subseteq E$ is a \emph{separator} if
\[
\rk E=\rk S+\rk(E\setminus S).
\]
Equivalently,
\[
M=(M|S)\oplus(M|(E\setminus S)).
\]
This is also equivalent to saying that no circuit meets both $S$ and $E\setminus S$. Consequently, the separators are exactly the unions of connected components.

\subsubsection{Oriented matroids}
An oriented matroid $\M$ on $E$ with underlying matroid $M$ is characterized here primarily through its covector set
\[
\Lcov(\M)\subseteq\{+,-,0\}^E,
\]
which has a partial order induced by $0<\pm$ coordinatewise. The maximal covectors are the \emph{topes}, denoted by $\T(\M)$. The oriented matroid $\M$ is also characterized by its \emph{chirotope}. If $r=\rk M$, a chirotope is an alternating map
\[
\chi:E^r\to\{+,-,0\}
\]
whose nonzero ordered supports are exactly the ordered bases of $M$, the sign records their orientation. For a covector $X\in\Lcov(\M)$, write its zero set as
\[
z(X)=\{e\in E \mid X_e=0\}.
\]
The zero sets of covectors are precisely $L(M)$. Restriction and contraction are compatible with the corresponding matroid operations,
\begin{align*}
\Lcov(\M|S)&=\{X|_S \mid X\in\Lcov(\M)\},\\
\Lcov(\M/S)&=\{X|_{E\setminus S} \mid X\in\Lcov(\M),\ X|_S=0\}.
\end{align*}
The covectors of a direct sum are the coordinatewise pairs of covectors of its two factors.

There are several useful operations on sign vectors. For a sign vector $X\in\{+,-,0\}^E$ and a subset $S\subseteq E$, let $X\setminus S$ and $\eps_S X$ be the sign vectors defined by
\begin{equation}\label{backslash-reversal}
(X\setminus S)_e=\begin{cases}
    X_e, & \text{ if }e\not\in S,\\
    0, & \text{ if }e\in S.
\end{cases}
\qquad
(\eps_S X)_e=\begin{cases}
    -X_e, &\text{ if } e\in S,\\
    X_e, &\text{ if } e\not\in S.
\end{cases}
\end{equation}
Also for $X,Y\in\{+,-,0\}^E$, their \emph{Tits product} $X\circ Y\in\{+,-,0\}^E$ is defined by
\begin{equation}\label{titsproduct}
(X\circ Y)_e=\begin{cases}
X_e, & \text{ if } X_e\neq 0,\\
Y_e, & \text{ if } X_e=0.
\end{cases}
\end{equation}
Their \emph{separating set} is
\[
\Sep_{\M}(X,Y)=\{e\in E\mid X_e\neq Y_e\}.
\]

\subsubsection{Hyperplane arrangements}
For a real central arrangement $\A=\{H_e \mid e\in E\}$ in $V$, choose defining forms $\varphi_e\in V^*$ with $H_e=\ker\varphi_e$. Then $\A$ defines an oriented matroid $\M$ whose covector set $\Lcov(\M)$ consists of sign vectors $\bigl(\sgn\varphi_e(x)\bigr)_{e\in E}$ for $x\in V$. We also say $\A$ (or $\{\varphi_e\}_{e\in E}$) is a realization of $\M$. A covector is also called a \emph{face} of $\A$. Geometrically, a face is the set of points $x$ with a common sign vector. In particular, the chambers of $\A$ correspond to the topes of $\M$. Throughout this paper, we frequently abuse the notions faces with covectors, and chambers with topes.

The chirotope of $\A$ is also useful in later sections. Put $W=\operatorname{span}_{\mathbb R}\{\varphi_e\mid e\in E\}\subseteq V^*$, with $r=\dim W=\rk\A$. After choosing an orientation of $W$, which is equivalent to choosing a positive side of the $1$-dimensional space $\det W:=\bigwedge^r W$, the chirotope of the arrangement is the alternating map
\[
\chi_{\A}:E^r\to \{+,-,0\}
\]
defined by
\[
\chi_{\A}(e_1,\ldots,e_r)=\begin{cases}
 +, & \varphi_{e_1}\wedge\cdots\wedge\varphi_{e_r}\text{ is positive in }\det W,\\
 -, & \varphi_{e_1}\wedge\cdots\wedge\varphi_{e_r}\text{ is negative in }\det W,\\
 0, & \varphi_{e_1},\ldots,\varphi_{e_r}\text{ are linearly dependent}.
 \end{cases}
\]
Its nonzero ordered supports are precisely the ordered bases of the underlying matroid. Reversing the orientation of $W$ multiplies the whole chirotope by $-1$, which does not change the oriented matroid. Replacing one defining form $\varphi_e$ by $-\varphi_e$ reorients the element $e$.

\subsection{Tope sets, tope graphs, and tope intervals}

For a loopless oriented matroid $\N$, its \emph{tope graph} $\G(\N)$ has vertex set $\T(\N)$ and two topes are adjacent if their separating set is a rank-one flat of the underlying matroid. If $\N$ is simple, every rank-one flat is a singleton, so
\[
A\sim B\iff|\Sep_\N(A,B)|=1.
\]
In this paper, we consider the tope graph $\G=\G(\A)=\G(\M)$ of a real central hyperplane arrangement $\A$. The underlying oriented matroid $\M$ is simple. Therefore,
\[
d_{\G}(A,B)=|\Sep_\M(A,B)|,
\]
and the sign-vector map embeds $\G$ isometrically into the hypercube $\{+,-\}^E$.

For a covector $X\in\Lcov(\M)$, put
\[
\T_X=\T(\M)_X=\{B\in\T(\M)\mid X\leq B\},
\]
and let $\G_X=\G(\M)_X$ be the subgraph of $\G(\M)$ induced by $\T(\M)_X$. We call it the \emph{facial subgraph} associated with $X$.

We define the closed interval $[A,B]_{\G}$ as the set of topes (including $A$ and $B$) traversed by any minimal walk in the graph $\G$ from $A$ to $B$. We equip $[A,B]_{\G}$ with a partial order $C\prec D$ if there is a minimal walk traversing $C$ before $D$. Also define the open interval $(A,B)_{\G}$ as obtained from $[A,B]_{\G}$ deleting the endpoints $A$ and $B$. A closed interval $[A,B]_{\G}$ is called \emph{facial} with respect to a covector $X\in\Lcov(\M)$ if $[A,B]_{\G}=\T(\M)_X$. The following form of the Edelman--Walker theorem will be used repeatedly.

\begin{theorem}[{\cite[Theorem 2.2]{Edelman1985}; see also \cite[Theorem 4.4.2]{Bjorner1999}}]\label{thm:EW}
Let $A,B\in\T(\M)$. If $[A,B]_{\G}$ is facial with respect to $X\in\Lcov(\M)$, then
\[
(A,B)_{\G}\simeq S^{\rk z(X)-2}.
\]
Otherwise $(A,B)_{\G}$ is contractible.
\end{theorem}

\subsection{Magnitude homology and cohomology}
See \cite{Hepworth2017,Hepworth2022} for details. Let $\X$ be a finite metric space. A proper $k$-chain in $\X$ is a tuple $\boldsymbol{x}=(x_0,x_1,\ldots,x_k)$ with $x_i\in \X$ and $x_{i-1}\neq x_i$ for every $i$. Its length is $\ell(\boldsymbol{x})=\sum_{i=1}^k d(x_{i-1},x_i)$.

Let $\kk$ be a ring. The magnitude chain group $\MC_{k,\ell}(\X;\kk)$ is the free $\kk$-module on the proper $k$-chains of length $\ell$. The differential is
\[
\partial=\sum_{i=1}^{k-1}(-1)^i\partial_i,
\]
where $\partial_i$ acting on $\boldsymbol{x}$ by deleting $x_i$ if $x_i$ is \emph{smooth}, that is,
\[
d(x_{i-1},x_{i+1})=d(x_{i-1},x_i)+d(x_i,x_{i+1}),
\]
and is zero otherwise. The homology of $(\MC_{*,\ell}(\X;\kk),\partial)$ is called the \emph{magnitude homology} of $\X$, denoted by $\MH_{*,\ell}(\X;\kk)$.

Taking $\Hom_\kk\bigl(\bullet,\kk\bigr)$ to $(\MC_{*,\ell}(\X;\kk),\partial)$, we obtain the magnitude cochain complex $(\MC^{*,\ell}(\X;\kk),\delta)$. The cohomology $\MH^{*,\ell}(\X;\kk)$ is called the \emph{magnitude cohomology} of $\X$. Magnitude cohomology is a stronger invariant than magnitude homology because it possess a product structure. For
\[
\varphi\in\MC^{p,\ell_1}(\X;\kk),\quad \psi\in\MC^{q,\ell_2}(\X;\kk),
\]
Hepworth's cup product $\varphi\smile\psi\in\MC^{p+q,\ell_1+\ell_2}(\X;\kk)$ is defined by
\begin{equation}\label{cup-chain}
(\varphi\smile\psi)(x_0,\ldots,x_{p+q})=\varphi(x_0,\ldots,x_p)\psi(x_p,\ldots,x_{p+q}).
\end{equation}
It descends to an associative unital product
\[
\MH^{p,\ell_1}(\X;\kk)\otimes\MH^{q,\ell_2}(\X;\kk)\xrightarrow{\smile}\MH^{p+q,\ell_1+\ell_2}(\X;\kk).
\]

\begin{proposition}[{\cite[Corollary 3.2]{Hepworth2022}}]\label{magcoh-is-complete}
If two finite metric spaces $\X$ and $\X'$ have $\MH^{*,*}(\X;\kk)\cong\MH^{*,*}(\X';\kk)$ as bigraded rings, then $\X$ and $\X'$ are isometric.
\end{proposition}

Since the differential $\partial$ preserves the endpoints of a proper chain, the proper chains with a fixed source $a\in\X$ and a fixed terminal $b\in\X$ span a subcomplex of $\MC_{*,\ell}(\X;\kk)$, denoted by $\MC_{*,\ell}(\X;\kk)_{a\to b}$. Similarly, we write $\MH_{*,\ell}(\X;\kk)_{a\to b}$ for its homology, $\MC^{*,\ell}(\X;\kk)_{a\to b}$ and $\MH^{*,\ell}(\X;\kk)_{a\to b}$ for the cochain complex and cohomology with fixed endpoints respectively. Notations like $\MH_{*,\ell}(\X;\kk)_{\bullet\to b}$ are understood similarly.

\section{Weighted matroid optimization and initial matroids}\label{matroid-optimization}

Standard references for weighted matroid optimization, the greedy theorem, and the base-polytope viewpoint are \cite{Oxley2011,Schrijver2003}. For initial (oriented) matroids and their relation with Bergman fans, see \cite{Ardila2006,Rau2022,Shaw2026}.

\subsection{Weights and initial matroids}

We fix a matroid $M$ on $E$. A \emph{weight}\footnote{This should not be confused with the weighting function in the context of magnitude.} is a function $w:E\to\R$, identified with its coordinate vector $w=(w_e)_{e\in E}\in\R^E$. It assigns to a subset $S\subseteq E$ the total weight $w(S)=\sum_{e\in S}w_e$. A basis $\bb\in\B(M)$ is \emph{$w$-maximizing} if
\[
w(\bb)=\max_{\cc\in\B(M)}w(\cc).
\]
We denote the set of $w$-maximizing bases by $\B_w(M)$, or simply by $\B_w$ when $M$ is known. If $P\subseteq\R^E$ is a polytope, the \emph{face of $P$ exposed by $w$} is
\[
F_w(P)=\{x\in P \mid w\cdot x=h_P(w)\}, \quad h_P(w)=\max_{y\in P}w\cdot y,
\]
where $\cdot$ is the usual dot product in $\R^E$. Thus the hyperplane $w\cdot x=h_P(w)$ supports $P$, and $F_w(P)$ is the part of $P$ touched by that supporting hyperplane. Let
\[
P(M)=\conv\{\one_{\bb} \mid \bb\in\B(M)\}\subseteq\R^E
\]
be the matroid base polytope. Then the face of $P(M)$ exposed by $w$ is
\[
F_w(P(M))=\conv\{\one_{\bb} \mid \bb\in\B_w(M)\}.
\]
In particular, the vertices of the exposed face are exactly the characteristic vectors of the $w$-maximizing bases. Every face of a matroid base polytope is again a matroid base polytope. Consequently, $\B_w$ is the basis set of a matroid on $E$.

\begin{definition}\label{weight-initial-matroid}
The matroid whose bases are $\B_w$ is the \emph{$w$-initial matroid of $M$}, denoted by $M_w$. If $\M$ is oriented with chirotope $\chi$, its \emph{oriented $w$-initial matroid} $\M_w$, has chirotope
\[
\chi_w(\bb)=\begin{cases}
  \chi(\bb),&\bb\in\B_w,\\
  0,&\bb\not\in\B_w.
\end{cases}
\]
Thus the original orientation is retained on the maximizing bases and all other bases disappear. Moreover,
\[
P(M_w)=F_w(P(M)).
\]
In particular, $M_w$ and $\M_w$ depend only on the exposed face $F_w(P(M))$, equivalently on the maximizing-basis set $\B_w$, and not on the particular vector $w$.
\end{definition}

An element is a loop of $M_w$ exactly when it occurs in no $w$-maximizing basis. We also record a useful criterion for nonloops.

\begin{lemma}\label{nonloop}
Let $\bb\in\B_w$ and $e\not\in\bb$. Write $\C_M(e,\bb)$ for the fundamental circuit. Then $e$ is not a loop of $M_w$ if and only if there is an element $f\in \C_M(e,\bb)\setminus\{e\}$ such that $w_f=w_e$.
\end{lemma}

\begin{proof}
If such $f$ exists, then $(\bb\setminus\{f\})\cup\{e\}$ is a basis of the same weight as $\bb$, so $e$ belongs to a basis of $M_w$. Conversely, let $\bb'\in\B_w$ contain $e$. By the strong basis-exchange property, there is $f\in\bb\setminus\bb'$ such that both $(\bb\setminus\{f\})\cup\{e\}$ and $(\bb'\setminus\{e\})\cup\{f\}$ are bases. The first exchange implies $f\in \C_M(e,\bb)$. Since $\bb$ and $\bb'$ are both $w$-maximizing, maximality applied to the two exchanged bases gives $w_f\geq w_e$ and $w_e\geq w_f$. Hence $w_f=w_e$.
\end{proof}

\subsection{Integral weights and $M$-admissibility}

Fix a matroid $M$ on $E$. For $w\in\NN^E$, define its \emph{upper level sets} by
\begin{equation}\label{upper-level-set}
U_j(w)=\{e\in E \mid w_e\geq j\}, \quad 1\leq j\leq m(w):=\max_{e\in E}w_e.
\end{equation}
They form a weakly decreasing sequence $U_1(w)\supseteq U_2(w)\supseteq\cdots \supseteq U_{m(w)}(w)$, and
\begin{equation}\label{layercake}
 w=\sum_{j=1}^{m(w)}\one_{U_j(w)}.
\end{equation}
Define a function $\kappa_M:\NN^E\to \NN$ by
\begin{equation}\label{kappa-def}
\kappa_M(w)=\sum_{j=1}^{m(w)}\rk_M U_j(w),\quad \kappa_M(\zero)=0.
\end{equation}

\begin{definition}
An integral weight $w\in\NN^E$ is \emph{$M$-admissible} if every nonempty $U_j(w)$ is a flat in $L(M)$.
\end{definition}

\begin{proposition}\label{kappa-support}
For every $w\in\NN^E$,
\[
\kappa_M(w)=\max_{\bb\in\B(M)}w(\bb).
\]
Consequently $\kappa_M$ is subadditive, that is, for $\alpha,\beta\in\NN^E$,
\[
\kappa_M(\alpha+\beta)\leq\kappa_M(\alpha)+\kappa_M(\beta),
\]
where equality holds if and only if $M$ has a basis which simultaneously maximizes $\alpha$ and $\beta$.
\end{proposition}

\begin{proof}
For every basis $\bb\in\B(M)$, \Cref{layercake} gives
\[
w(\bb)=\sum_{j\geq 1}|\bb\cap U_j(w)|\leq\sum_{j\geq 1}\rk U_j(w)=\kappa_M(w).
\]
The equality can be attained as follows. Since $E\supseteq U_1(w)\supseteq U_2(w)\supseteq\cdots \supseteq U_{m(w)}(w)$, start from a basis of $U_{m(w)}(w)$, extend it successively to a basis $\bb$ of $M$. Then $|\bb\cap U_j(w)|=\rk U_j(w)$ for every $j$, attaining the equality. The second claim then follows immediately.
\end{proof}

We say that a basis $\bb\in\B(M)$ \emph{saturates} $S\subseteq E$ if $\bb\cap S$ is a basis for $S$, that is $|\bb\cap S|=\rk_M S$.

\begin{definition}\label{w-tight}
A subset $S\subseteq E$ is \emph{$w$-tight} if every $w$-maximizing basis saturates $S$, in other words,
\[
|\bb\cap S|=\rk_M S\qquad\text{for every }\bb\in\B_w.
\]
Equivalently, the exposed face $F_w(P(M))$ is contained in the hyperplane $\sum_{e\in S}x_e=\rk_M S$.
\end{definition}

\subsection{Flag-initial oriented matroids}

Let $\M$ be an oriented matroid on $E$ with underlying matroid $M$. Let
\[
\F:\quad\varnothing=F_0\subsetneq F_1\subsetneq\cdots\subsetneq F_s=E
\]
be a strict flag of flats in $L(M)$. Its \emph{flag-initial oriented matroid} is defined by
\begin{equation}\label{initial-OM}
\M_{\F}=\bigoplus_{i=1}^s(\M|F_i)/F_{i-1}.
\end{equation}
Geometrically, this is the associated graded normal arrangement along the flag. The $i$-th summand records the new normal directions contributed by $F_i\setminus F_{i-1}$. For this flag $\F$, let
\[
\T(\M)_\F=\{B\in\T(\M) \mid B\setminus F_i\in\Lcov(\M)\text{ for every }i\}
\]
the set of topes of $\M$ incident to every flat of the flag, see \Cref{backslash-reversal} for notations. Then it is known (see \cite[Lemma 3.1]{Rau2022} and \cite[Proposition 3.5]{Shaw2026}) that
\begin{equation}\label{initial-topes}
\T(\M_{\F})=\T(\M)_{\F}.
\end{equation}
The next lemma identifies the weight-initial and flag-initial constructions for the special integral weights used in this paper.

\begin{lemma}\label{admissible-initial-loopless}
Let $w\in\NN^E$ be $M$-admissible. Remove repetitions from its nonempty level sets $U_j(w)$, list them in increasing order, and append the endpoints when necessary to form a strict flag of flats
\[
\F_w:\quad \varnothing=F_0\subsetneq F_1\subsetneq\cdots\subsetneq F_s=E.
\]
Then
\[
\M_w=\bigoplus_{i=1}^s(\M|F_i)/F_{i-1}=\M_{\F_w}.
\]
In particular, $M_w$ is loopless and
\[
\T(\M_w)=\{B\in\T(\M) \mid B\setminus F_i\in\Lcov(\M)\text{ for all }i\}.
\]
\end{lemma}

\begin{proof}
Put $S_i=F_i\setminus F_{i-1}$ and $r_i=\rk F_i-\rk F_{i-1}$. The weight $w$ is constant on each $S_i$, with value denoted by $d_i$. Then $d_1>d_2>\cdots>d_s\geq 0$ and
\[
w=d_s\one_E+\sum_{i=1}^{s-1}(d_i-d_{i+1})\one_{F_i}.
\]
Thus for every basis $\bb\in\B(M)$,
\[
w(\bb)=d_s\rk M+ \sum_{i=1}^{s-1}(d_i-d_{i+1})|\bb\cap F_i|.
\]
Since $|\bb\cap F_i|\leq\rk F_i$ and all coefficients $d_i-d_{i+1}$ are positive, the basis $\bb$ is $w$-maximizing if and only if $|\bb\cap F_i|=\rk F_i$ for all $i$. Equivalently, putting $\bb_i=\bb\cap S_i$, each $\bb_i$ is a basis of $(M|F_i)/F_{i-1}$. It follows that
\[
M_w=\bigoplus_{i=1}^s(M|F_i)/F_{i-1}.
\]
Each summand $(M|F_i)/F_{i-1}$ is loopless because $F_{i-1}$ is a flat of $M|F_i$. Hence $M_w$ is loopless.

It remains to compare the orientations. Let $(v_e)_{e\in E}\subseteq V$ be a vector configuration realizing $\M$, and put $W_i=\operatorname{span}\{v_e\mid e\in F_i\}$. For $e\in S_i$, let $\bar v_e=v_e+W_{i-1}\in W_i/W_{i-1}$. The configuration $(\bar v_e)_{e\in S_i}$ realizes $(\M|F_i)/F_{i-1}$, so their direct sum in $\bigoplus_{i=1}^sW_i/W_{i-1}$ realizes $\M_{\F_w}$.

If $\bb$ is not $w$-maximizing, then it is not a basis of this direct sum, so both $\chi_w(\bb)$ and $\chi_{\M_{\F_w}}(\bb)$ vanish. If $\bb$ is $w$-maximizing, order it blockwise as $\bb=(\bb_1,\ldots,\bb_s)$. Under the canonical determinant-line isomorphism $\det V\cong\bigotimes_{i=1}^s\det(W_i/W_{i-1})$, one has
\[
\bigwedge_{e\in\bb_1}v_e\wedge\cdots\wedge\bigwedge_{e\in\bb_s}v_e\longmapsto\bigotimes_{i=1}^s\bigwedge_{e\in\bb_i}\bar v_e.
\]
After transporting the orientation through this isomorphism, the two determinants have the same sign. Hence
\[
\chi_{\M_{\F_w}}(\bb)=\chi_{\M}(\bb)=\chi_w(\bb).
\]
Thus $\M_w=\M_{\F_w}$. Finally the description of the topes follows from \eqref{initial-topes}.
\end{proof}

\begin{example}\label{weight-U23}
Let $M=U_{2,3}$ be the uniform matroid of rank $2$ on $\{1,2,3\}$ and $w=(2,1,1)$. The three bases have weights
\[
w(\{1,2\})=3,\quad w(\{1,3\})=3,\quad w(\{2,3\})=2.
\]
Thus $M_w$ has bases $\{1,2\}$ and $\{1,3\}$. Equivalently, element $1$ is a coloop and $2,3$ form a parallel pair,
\[
M_w\cong U_{1,\{1\}}\oplus U_{1,\{2,3\}}.
\]
Since $w$ is $U_{2,3}$-admissible, this is also the underlying matroid of the flag-initial oriented matroid for $\varnothing\subset\{1\}\subset E$.
\end{example}

\subsection{The tope graph of an initial oriented matroid}
\label{initial-tope-graph}

For a flag of flats $\F$, the initial-tope description \eqref{initial-topes} gives a canonical inclusion $\T(\M_{\F})\hookrightarrow\T(\M)$. It is the identity on sign vectors. In general, this does \emph{not} identify $\G(\M_{\F})$ with the subgraph of $\G(\M)$ induced by $\T(\M_{\F})$.

The direct sum definition of the initial oriented matroid gives
\[
\T(\M_{\F})\cong\prod_{i=1}^s\T\bigl((\M|F_i)/F_{i-1}\bigr)
\]
and at the graph level,
\[
\G(\M_{\F})\cong\mathop{\square}_{i=1}^s\G\bigl((\M|F_i)/F_{i-1}\bigr),
\]
where $\square$ denotes the Cartesian product of graphs. An edge of $\G(\M_{\F})$ can reverse the signs on an entire rank-one flat of the initial matroid. That flat may contain several elements of the original ground set. Its two endpoints are then separated in the original arrangement by every element of this block and need not be adjacent in $\G(\M)$.

For a complete flag
\[
\F:\quad \varnothing=F_0\subsetneq F_1\subsetneq\cdots\subsetneq F_r=E,\quad S_i=F_i\setminus F_{i-1},
\]
each $S_i$ is the ground set of a rank-one direct summand $(\M|F_i)/F_{i-1}$ of $\M_{\F}$. Consequently, $\G(\M_{\F})\cong Q_r$ the $r$-dimensional hypercube graph. After choosing one incident tope $A\in\T(\M_\F)$, write
\[
A_J=\eps_{\bigcup_{i\in J}S_i}A,\quad J\subseteq[r],
\]
the tope opposite to $A$ with respect to $\bigcup_{i\in J}S_i$ (see notations \Cref{backslash-reversal}). The two relevant metrics on this common vertex set are
\[
d_{\G(\M_{\F})}(A_J,A_K)=|J\mathbin{\triangle}K|,\quad d_{\G(\M)}(A_J,A_K)=\sum_{i\in J\mathbin{\triangle}K}|S_i|,
\]
where $\mathbin{\triangle}$ is the symmetric difference. Thus the same set of topes is an ordinary cube for $\G(\M_\F)$ and a weighted cube for $\G(\M)$. In particular, an edge of $\G(\M_\F)$ in the $i$-th coordinate has ambient length $|S_i|$.

\begin{figure}[!ht]
\centering
\begin{minipage}[c]{0.30\textwidth}
\centering
\begin{tikzpicture}[scale=0.72]
  \draw[->] (-2.2,0)--(2.2,0) node[right] {$H_2$};
  \draw[->] (0,-2.2)--(0,2.2) node[above] {$H_1$};
  \draw (-1.8,1.8)--(1.8,-1.8) node[below right] {$H_3$};
  \node at (0.75,0.75) {${+++}$};
  \node at (-0.75,1.20) {${-++}$};
  \node at (-1.25,0.35) {${-+-}$};
  \node at (-0.75,-0.75) {${---}$};
  \node at (0.70,-1.20) {${+--}$};
  \node at (1.25,-0.35) {${+-+}$};
\end{tikzpicture}
\end{minipage}\hfill
\begin{minipage}[c]{0.30\textwidth}
\centering
\begin{tikzpicture}[scale=0.82,every node/.style={font=\scriptsize}]
  \foreach \i/\lab in {0/+++ ,1/-++ ,2/-+- ,3/--- ,4/+-- ,5/+-+}{
    \coordinate (v\i) at ({90+60*\i}:1.55);
    \fill (v\i) circle (1.8pt);
    \node at ({90+60*\i}:1.95) {$\lab$};
  }
  \foreach \i in {0,...,5}{\pgfmathtruncatemacro{\j}{mod(\i+1,6)}\draw (v\i)--(v\j);}
\end{tikzpicture}
\end{minipage}\hfill
\begin{minipage}[c]{0.35\textwidth}
\centering
\begin{tikzpicture}[x=1.25cm,y=1.05cm,every node/.style={font=\scriptsize}]
  \node (a) at (0,1) {$+++$};
  \node (b) at (1,1) {$-++$};
  \node (c) at (0,0) {$+--$};
  \node (d) at (1,0) {$---$};
  \draw (a)--node[above] {$S_1=\{1\}$}(b);
  \draw (c)--(d);
  \draw[dashed] (a)--node[left] {$S_2=\{2,3\}$}(c);
  \draw[dashed] (b)--(d);
\end{tikzpicture}
\end{minipage}
\caption{The arrangement $\A=\{x=0,y=0,x+y=0\}$; its tope graph $\G(\M)$, a six-cycle; and its initial tope graph $\G(\M_\F)$ for the complete flag $\F:\varnothing\subset\{1\}\subset E$. A dashed side reverses the block $S_2=\{2,3\}$ and therefore has ambient tope distance two.}
\label{three-line-initial}
\end{figure}
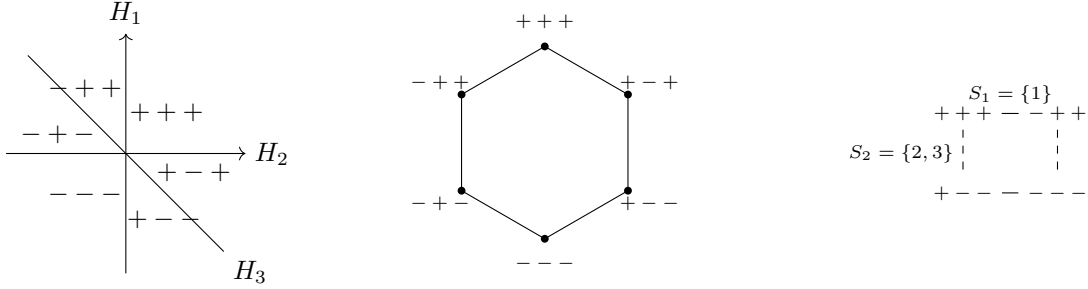

\begin{example}\label{three-line-minors}
Consider the arrangement $\A=\{H_1,H_2,H_3\}$ defined by $\varphi_1=x$, $\varphi_2=y$, and $\varphi_3=x+y$. The oriented matroid $\M$ has underlying matroid $U_{2,3}$. For the flag of flats
\[
\F:\quad \varnothing\subset\{1\}\subset E,
\]
with $S_1=\{1\}$ and $S_2=\{2,3\}$. The initial oriented matroid is
\[
\M_{\F}=(\M|\{1\})\oplus(\M/\{1\}).
\]
In the contraction summand, the images of $\varphi_2$ and $\varphi_3$ coincide, so $2$ and $3$ become positively parallel. Then $\T(\M_{\F})=\{(+++),(-++),(+--),(---)\}$. The initial tope graph is $\G(\M_{\F})\cong Q_2$, the Boolean square on the right of \Cref{three-line-initial}. The two $S_2$-labeled edges have length two in the ambient graph $\G(\M)$ and are absent from the subgraph induced on these four vertices. That induced subgraph is therefore the disjoint union of two edges, whereas $\G(\M_{\F})$ is a square. Thus $\T(\M_{\F})$ is canonically a subset of $\T(\M)$, whereas $\G(\M_{\F})$ is in general not a subgraph of $\G(\M)$.
\end{example}

\subsection{Sorting subset sequences}

For the ground set $E$, consider a sequence of subsets $\bS=(S_1,\ldots,S_m)$. An elementary sorting move on $\bS$ is to choose an adjacent pair $(S_i,S_{i+1})$ of subsets with $S_i\not\supseteq S_{i+1}$ and perform the replacement
\begin{equation}\label{sorting-move}
(S_i,S_{i+1})\longmapsto(S_i\cup S_{i+1},S_i\cap S_{i+1}).
\end{equation}

We first have an elementary observation.
\begin{lemma}\label{set-sorting}
Repeated application of \eqref{sorting-move} to $\bS=(S_1,\ldots,S_m)$ terminates. If $w=\sum_j\one_{S_j}$, then the terminal sequence (with empty sets in the tail omitted) is the nested upper level sets \eqref{upper-level-set} of $w$,
\[
U_1(w)\supseteq U_2(w)\supseteq\cdots\supseteq U_{m(w)}(w).
\]
\end{lemma}

\begin{proof}
For a sequence $\bS=(S_1,\ldots,S_m)$, define the nonnegative integer
\[
\nu(S_1,\ldots,S_m)=\sum_{j=1}^m j|S_j|.
\]
Suppose $S_i\not\supseteq S_{i+1}$. After performing the sorting move for $(S_i,S_{i+1})$, the change in $\nu$ is
\[
i|S_i\cup S_{i+1}|+(i+1)|S_i\cap S_{i+1}|-i|S_i|-(i+1)|S_{i+1}|=-|S_{i+1}\setminus S_i|<0.
\]
Hence $\nu$ strictly decreases at every move, and the process terminates. At a terminal sequence no adjacent pair is eligible for sorting. Therefore it is nested
\[
H_1\supseteq H_2\supseteq\cdots\supseteq H_m.
\]
The identity $\one_{S_i}+\one_{S_{i+1}}=\one_{S_i\cup S_{i+1}}+\one_{S_i\cap S_{i+1}}$ shows that every sorting move preserves the vector $w=\sum_j\one_{S_j}$. For a fixed element $e\in E$, nestedness implies that
\[
e\in H_j\iff w_e\geq j,
\]
and hence $H_j=U_j(w)$.
\end{proof}

\subsection{The tight-flat lattice and tope incidence}

Let $F_1,\ldots,F_m$ be flats of $M$ and $w=\sum_{i=1}^m\one_{F_i}\in\NN^E$. By \Cref{kappa-support}, for any $\bb\in\B_w$,
\[
\kappa_M(w)=w(\bb)=\sum_{i=1}^m|\bb\cap F_i|\leq\sum_{i=1}^m\rk F_i.
\]
If we assume $\kappa_M(w)=\sum_{i=1}^m\rk F_i$, then 
\begin{equation}\label{flat-is-tight}
|\bb\cap F_i|=\rk F_i,\quad \text{ for every } i,
\end{equation}
meaning that $F_i$ is $w$-tight (\Cref{w-tight}). The next proposition says much more. Provided $M_w$ has no loops, $w$-tightness extends through unions and intersections. The resulting sets remain flats of the original matroid, and inside $M_w$ they are separators. We adjoin $\varnothing$ and $E$ whenever we speak of the lattice generated by a family of subsets.

\begin{proposition}\label{tight-flat-lattice}
Let $F_1,\ldots,F_m$ be flats of $M$, put $w=\sum_{i=1}^m\one_{F_i}$. We assume that $M_w$ is loopless, and
\begin{equation}\label{tight-degree}
\kappa_M(w)=\sum_{i=1}^m\rk F_i.
\end{equation}
Let $\D$ be the smallest family of subsets of $E$ containing $F_1,\ldots,F_m,\varnothing,E$ and closed under union and intersection. Then,
\begin{enumerate}[label=\textup{(\roman*)}]
\item every $S\in\D$ is $w$-tight;
\item every $S\in\D$ is a flat of $M$;
\item every $S\in\D$ is a separator of $M_w$;
\item every pair $S,T\in\D$ is a closed modular pair:
\[
\rk S+\rk T=\rk(S\cup T)+\rk(S\cap T),\quad S\cup T\text{ is a flat}.
\]
\end{enumerate}
\end{proposition}

This proposition is the technical core of the our argument. The impact of these conclusions will be clear as the argument proceeds. Part~(ii) ensures that every set created by sorting still indexes a single-flat cohomology class. Part~(i) keeps the degree unchanged. Part~(iii) allows independent sign reversal on the set inside the initial oriented matroid. Part~(iv) is exactly the rank equality needed to apply the local modular relation.

\begin{proof}
(i) For $a,b\in E$, write
\begin{equation}\label{membership-preorder}
a\preceq b \iff a\in F_i \text{ implies } b\in F_i \text{ for every }i.
\end{equation}
This relation compares only the membership patterns of $a$ and $b$ in the original flats $F_i$. Since $\D$ is generated from the $F_i$ using unions and intersections, an elementary induction gives
\begin{equation}\label{membership-transfer}
a\preceq b,\quad a\in S,\quad S\in\D\quad\Longrightarrow\quad b\in S.
\end{equation}
Indeed, the implication is immediate for $S=F_i$, $S=\varnothing$, and $S=E$, and it is preserved when two sets satisfying it are replaced by their union or their intersection.

Take $S\in\D$ and $\bb\in\B_w$. We first show that $\bb\cap S$ spans $S$. Let $e\in S\setminus\bb$, and put $\C=\C_M(e,\bb)$. We claim that $e\preceq f$ for every $f\in\C\setminus\{e\}$. Indeed, fix $i$. If $e\notin F_i$, there is nothing to prove. If $e\in F_i$, then $\bb\cap F_i$ is a basis of the flat $F_i$ by \eqref{flat-is-tight}, and hence it spans $e$. Apply \Cref{fund-circuit-lemma} for $I=\bb\cap F_i$ and we obtain $\C\subseteq(\bb\cap F_i)\cup\{e\}$, so $f\in F_i$. Thus $e\preceq f$. Since $e\in S$ and $S\in\mathcal D$, implication \eqref{membership-transfer} gives $f\in S$. Therefore,
\[
\C_M(e,\bb)\setminus\{e\}\subseteq\bb\cap S,
\]
and hence $e\in\cl_M(\bb\cap S)$. This holds for every $e\in S\setminus\bb$, so $\bb\cap S$ spans $S$. It is independent because it is contained in the basis $\bb$, and hence it is a basis of $S$, that is $\rk S=|\bb\cap S|$.

(ii) We next prove that $S$ is a flat. Suppose that $e\in\cl_M(S)\setminus S$. Since $\bb\cap S$ is a basis of $S$, one has $e\notin\bb$ and again by \Cref{fund-circuit-lemma} for $I=\bb\cap S$,
\[
\C_M(e,\bb)\setminus\{e\}\subseteq\bb\cap S.
\]
For every $b$ in this fundamental circuit, the argument above again gives $e\preceq b$. Moreover, $b\in S$ whereas $e\notin S$. Their membership patterns in the $F_i$ therefore cannot be equal. Hence there is at least one index $i$ such that $e\notin F_i$ and $b\in F_i$. Since $e\preceq b$, it follows that
\[
w_b=\sum_{i=1}^m\one_{F_i}(b)>\sum_{i=1}^m\one_{F_i}(e)=w_e, \quad \text{for every }b\in \C_M(e,\bb)\setminus\{e\}.
\]
By \Cref{nonloop}, $e$ is a loop of $M_w$, contrary to the hypothesis that $M_w$ is loopless. Therefore $S=\cl_MS$.

(iii) Put $N=M_w$. Recall that $\B(N)=\B_w\subseteq\B(M)$. Hence every $N$-independent set is $M$-independent, and we have $\rk_N S\leq\rk_M S$ for $S\subseteq E$.

Fix $S\in\D$ and an $N$-basis $\bb\in\B_w$. We proved in (i) that $|\bb\cap S|=\rk_M S$. On the other hand, $\bb\cap S$ is $N$-independent because it is contained in the $N$-basis $\bb$. Therefore $\rk_N S\geq|\bb\cap S|=\rk_M S$. Together with the reverse inequality, this gives $\rk_N S=\rk_M S$. In particular, every $N$-basis $\bb\in\B_w$ satisfies
\begin{equation}\label{saturation-on-N}
|\bb\cap S|=\rk_N S. 
\end{equation}
We claim that
\begin{equation}\label{S-is-separator}
\rk_N(E\setminus S)=\rk_N E-\rk_N S.    
\end{equation}
The inequality $\geq$ follows by taking the complement of $\bb\cap S$ in any $N$-basis $\bb$.  If the inequality were strict, there would be an $N$-independent set $I\subseteq E\setminus S$ with $|I|>\rk_N E-\rk_N S$. Extending $I$ to an $N$-basis $\bb'$ would give
\[
|\bb'\cap S|=\rk_N E-|\bb'\cap(E\setminus S)|<\rk_N S,
\]
contrary to \Cref{saturation-on-N}. Hence \Cref{S-is-separator} is proved, which means $S$ is a separator of $N$.

(iv) Finally, choose $\bb\in\B_w$. Applying (i) to $S,T,S\cup T,S\cap T$ gives
\begin{align*}
\rk S+\rk T&=|\bb\cap S|+|\bb\cap T|\\
 &=|\bb\cap(S\cup T)|+|\bb\cap(S\cap T)|\\
 &=\rk(S\cup T)+\rk(S\cap T).
\end{align*}
The union is a flat by (ii), proving (iv).
\end{proof}

\begin{example}\label{tight-U23}
Return to $M=U_{2,3}$ and $w=(2,1,1)$ from \Cref{weight-U23}. Here
\[
w=\one_E+\one_{\{1\}},\quad \kappa_M(w)=\rk E+\rk\{1\}=3.
\]
The lattice generated by $E$ and $\{1\}$ is
\[
\{\varnothing,\{1\},E\}.
\]
In $M_w=U_{1,\{1\}}\oplus U_{1,\{2,3\}}$, the set $\{1\}$ is a whole direct sum component and hence a separator. Thus reversing the sign on coordinate $1$ preserves the initial tope set. This is the smallest example of the mechanism used in the next lemma.
\end{example}

We also need an oriented refinement. Recall from \Cref{weight-initial-matroid} that the initial chirotope $\chi_w$ agrees with the chirotope of $\M$ on the $w$-maximizing bases and vanishes on all other bases. In particular, if $w,w'\in\NN^E$ have the same maximizing bases $\B_w=\B_{w'}$, then $\M_w=\M_{w'}$.

The next lemma provides the oriented matroid information needed during sorting. Its first conclusion says that a separator can be reoriented independently inside the direct sum $\M_w$. Its second conclusion says that every tope of $\M_w$ is incident, in the original oriented matroid $\M$, to every tight flat produced by unions and intersections.

\begin{lemma}\label{tight-tope-incidence}
Under the hypotheses of \Cref{tight-flat-lattice}, assume that $w$ is $M$-admissible and let $\N=\M_w$. For every $S\in\D$ and every tope $A\in\T(\N)$ one has $\eps_SA\in\T(\N)$ and $A\setminus S\in\Lcov(\M)$.
\end{lemma}

\begin{proof}
By \Cref{tight-flat-lattice}, $S$ is a separator of the underlying matroid of $\N$. Hence reversing all signs on $S$ preserves its tope set. This proves the first assertion.

For the second, the idea is to replace $w$ by another $M$-admissible weight without changing the initial oriented matroid, but has $S$ as an upper level set. This can be done by the following perturbation. Choose an integer $L$ sufficiently large, say $L>n:=\max_{e\in E}w_e$. Put
\[
w'=w+L\cdot\one_S.
\]
We claim that $\B_w=\B_{w'}$ and hence $\M_{w'}=\M_w=\N$. In fact, for $\cc\in\B(M)$,
\[
w'(\cc)=w(\cc)+L\cdot|\cc\cap S|\leq W+L\cdot\rk_MS,
\]
where $W=\max_{\mathfrak c\in\B(M)}w(\mathfrak c)=\kappa_M(w)$. By \Cref{tight-flat-lattice}, $\bb\in\B_w$ satisfies $|\bb\cap S|=\rk_MS$. Therefore it is $w'$-maximizing $w'(\bb)=w(\bb)+L\cdot|\bb\cap S|= W+L\cdot\rk_MS$. This proves $\B_w\subseteq \B_{w'}$. Conversely, if $\cc\in\B(M)$ saturates $S$, that is $|\cc\cap S|=\rk_MS$, then it is $w'$-maximizing exactly when it is $w$-maximizing. As for those $\cc$ that do not saturate $S$, they cannot be $w'$-maximizing since the above inequality is strict.

If $f\in S$ and $g\in E\setminus S$, we have
\begin{equation}\label{gap}
w'_f=w_f+L>n\geq w_g=w'_g.    
\end{equation}
Then $S$ becomes an upper level set of $w'$, say $S=U_{1+n}(w')$.

We also need to prove that $w'$ is $M$-admissible. Since $w=\sum_i\one_{F_i}$, we have
\[
U_j(w)=\bigcup_{\substack{J\subseteq[m]\\|J|=j}}\bigcap_{i\in J}F_i\in \D.
\]
Because of the gap \eqref{gap}, every upper level set $U_i(w')$ is of the form $S\cap U_j(w),S$ or $S\cup U_j(w)$, and hence belongs to $\D$. By \Cref{tight-flat-lattice}, $U_i(w')$ is a flat of $M$, so $w'$ is $M$-admissible and $S$ occurs in its level-flat flag. Since $A\in\T(\M_{w'})$, the initial tope characterization \eqref{initial-topes} gives $A\setminus S\in\Lcov(\M)$.
\end{proof}

\section{Periodicity of magnitude cohomology}\label{periodicty-section}

We now turn to the discussion of magnitude (co)homology of the tope graph $\G=\G(\A)$ of an arrangement $\A$. Let $\M$ be the underlying oriented matroid and $M$ the underlying matroid, with ground set $E$.

\subsection{Crossing vectors and tope-admissibility}

In the context of magnitude (co)homology of tope graphs, the length grading can be refined by crossing vectors.

For two topes $A,B\in\T=\T(\M)$, define the separating vector
\[
\vec{d}(A,B)=\one_{\Sep(A,B)}\in\{0,1\}^E.
\]
For a proper chain $\bA=(A_0,\ldots,A_k)$ in $\G$, put
\[
\vec{\ell}(\bA)=\sum_{i=1}^k\vec{d}(A_{i-1},A_i)\in\NN^E.
\]
The vector $\vec{\ell}(\bA)$ records how many times the chain crosses each hyperplane and is called its \emph{crossing vector}. If $A_i$ is smooth in $\bA$, then deleting it preserves the crossing vector $\vec{\ell}(\bA)=\vec{\ell}(\partial_i\bA)$. Hence we denote the subcomplex of $\MC_{*,\ell}(\G;\kk)$ spanned by proper chains with a prescribed crossing vector $\alpha\in\NN^E$ with $|\alpha|=\sum_{e\in E}\alpha_e=\ell$ by
\[
\MC_{*,\alpha}(\G;\kk),
\]
and its homology by $\MH_{*,\alpha}(\G;\kk)$. Similarly, we denote by $\MC^{*,\alpha}(\G;\kk)$ and $\MH^{*,\alpha}(\G;\kk)$ the cochain complex and coholomogy with crossing grading $\alpha$. Then the cup product satisfies
\[
\MH^{p,\alpha}(\G;\kk)\otimes\MH^{q,\beta}(\G;\kk)\xrightarrow{\smile}\MH^{p+q,\alpha+\beta}(\G;\kk).
\]

For $S\subseteq E$ and a sign vector $B$, recall from \Cref{backslash-reversal} that $\eps_SB$ is obtained by reversing the signs on $S$. For $\alpha\in\NN^E$, abbreviate
\[
\eps_\alpha B:=\eps_{\{e\in E\mid \alpha_e\text{ is odd}\}}B.
\]
Also note that $\eps_{\alpha}^2=\id$. Every proper chain of crossing vector $\alpha$ ending at $B$ begins at $\eps_\alpha B$. Thus
\[
\MC_{*,\alpha}(\G;\kk)_{A\to B}=0 \text{ unless } A=\eps_\alpha B.
\]

\begin{definition}
Let $\alpha\in\NN^E$ and $B\in\T(\M)$. The pair $(\alpha,B)$ is \emph{tope-admissible} if $\alpha$ is $M$-admissible and $B\setminus U_j(\alpha)\in\Lcov(\M)$ for every $j$. Equivalently, $B$ is incident to every flat in the level-flat flag $\F_\alpha$ of $\alpha$.
\end{definition}

By \eqref{initial-topes} and \Cref{admissible-initial-loopless}, the allowable terminal topes are exactly the topes of the corresponding initial oriented matroid
\begin{equation}\label{eq:admissible-initial-topes}
\{B\in\T(\M)\mid(\alpha,B)\text{ is tope-admissible}\}=\T(\M_{\F_\alpha})=\T(\M_\alpha),
\end{equation}

\subsection{Face flag theorem}
A face flag relative to a tope $B\in\T$ is defined as
\[
\Phi:\quad X_1\geq X_2\geq\cdots\geq X_m, \quad X_i\in\Lcov(\M)_{<B},
\]
where $m\geq 0$. Define its rank and crossing vector by
\[
\rk\Phi=\sum_{i=1}^m\rk z(X_i)\in\NN,\quad \vec{\ell}(\Phi)=\sum_{i=1}^m\one_{z(X_i)}\in\NN^E.
\]

\begin{lemma}\label{unique-flag}
Fix $\alpha\in\NN^E$ and $B\in\T(\M)$. There is at most one face flag $\Phi$ relative to $B$ with crossing vector $\alpha$. It exists if and only if $(\alpha,B)$ is tope-admissible. When it exists, its zero flats, with repetitions, are
\[
U_{m(\alpha)}(\alpha)\subseteq\cdots\subseteq U_1(\alpha).
\]
Its rank $\rk\Phi$ is $\kappa_M(\alpha)$.
\end{lemma}

\begin{proof}
Suppose the zero flats of a face flag form the weakly increasing sequence $F_1\subseteq\cdots\subseteq F_m$ and $\alpha=\sum_i\one_{F_i}$. Counting the number of sets containing a fixed element gives
\[
F_{m-j+1}=U_j(\alpha).
\]
Thus the sequence is forced. For a fixed tope $B\in\T$, its only possible face with zero flat $F$ is $B\setminus F$, and it exists precisely when this sign vector is a covector. The rank formula follows from \eqref{kappa-def}.
\end{proof}

The following is a refinement of Theorem 6.4 (1) of \cite{Koizumi2026a}.
\begin{theorem}\label{Koizumi-basis}
For every $\alpha\in\NN^E$ and $B\in\T$,
\[
\MH_{k,\alpha}(\G;\ZZ)_{\bullet\to B}\cong\begin{cases}
 \ZZ,&k=\kappa_M(\alpha)\text{ and }(\alpha,B)\text{ is tope-admissible},\\
 0,&\text{otherwise}.
 \end{cases}
\]
In particular, all integral magnitude homology groups of $\G$ are torsion-free.
\end{theorem}

\begin{proof}
Theorem 6.4 of \cite{Koizumi2026a} proves that $\MH_{k,\alpha}(\G;\ZZ)_{\bullet\to B}$ is free abelian on the face flags $\Phi$ relative to $B$ with $\rk\Phi=k$ and $\vec{\ell}(\Phi)=\alpha$.  The claim then follows from \Cref{unique-flag}.
\end{proof}

\begin{remark}
\cite[Corollary 6.7]{Koizumi2026a} proves that $\MH_{*,*}(\G;\ZZ)$ is determined by the intersection lattice $L(\A)$. On the other hand, \Cref{magcoh-is-complete} shows that the magnitude cohomology ring $\MH^{*,*}(\G;\kk)$ determines the graph $\G$ up to isomorphisms, and hence cannot be determined by solely $L(\A)$.
\end{remark}

Because the integral homology is torsion-free, mod $2$ cohomology is canonically dual to mod $2$ homology. Thus every tope-admissible pair determines a canonical class.

\begin{definition}\label{u-class}
For a tope-admissible pair $(\alpha,B)$ of $\M$, let
\[
u_{\alpha,B}\in\MH^{\kappa_M(\alpha),\alpha}(\G;\Ftwo)_{\eps_\alpha B\to B}
\]
be the unique nonzero element. For $\alpha=\zero$, write $e_B=u_{\zero,B}$.
\end{definition}

The $e_B$ are orthogonal idempotents and $1=\sum_Be_B$. Moreover,
\[
e_A\smile u_{\alpha,B}=\delta_{A,\eps_\alpha B}u_{\alpha,B}, \qquad  u_{\alpha,B}\smile e_C=\delta_{B,C}u_{\alpha,B},
\]
where $\delta_{*,*}$ is the Kronecker delta.

\begin{example}
Continue \Cref{three-line-minors} and take $\alpha=(2,1,1)$. Then
\[
U_1(\alpha)=E,\quad U_2(\alpha)=\{1\}, \quad  \kappa_M(\alpha)=\rk E+\rk\{1\}=2+1=3.
\]
The vector is $U_{2,3}$-admissible. Its tope-admissible terminal topes are exactly the four topes in the initial square of \Cref{three-line-initial}. Hence $\MH_{3,(2,1,1)}(\G;\ZZ)\cong\ZZ^4$. For example, $B=(+++)$ has source $\eps_\alpha B=(+--)$, and the corresponding canonical class $u_{(2,1,1),(+++)}\in\MH^{3,(2,1,1)}(\G;\Ftwo)_{(+--)\to(+++)}$.
\end{example}

\subsection{Special chains}

In this subsection, let us fix $\alpha,\beta\in\NN^E$ and $C\in\T$, and put
\[
B=\eps_\beta C, \quad A=\eps_\alpha B
\]
For a chain
\[
\bD=(A=D_0,D_1,\ldots,D_k=C)\in\MC_{k,\alpha+\beta}(\G;\kk)_{A\to C},
\]
define the \emph{cumulative crossing vectors}
\[
\lambda_i(\bD)=\sum_{j=1}^i\vec{d}(D_{j-1},D_j)\in\NN^E, \quad 0\leq i\leq k.
\]
The vectors $\lambda_i(\bD)$ are strictly increasing in the coordinatewise order
\[
\zero=\lambda_0(\bD)<\lambda_1(\bD)<\cdots<\lambda_k(\bD)=\alpha+\beta.
\]

Next we generalize the notion of special chains defined in Definition 5.1 of \cite{Koizumi2026a}.

\begin{definition}
The chain $\bD$ is \emph{$\alpha$-special} if $\lambda_i(\bD)=\alpha$ for some $i$. Let $N^\alpha_{*,\alpha+\beta}(\G;\kk)_{A\to C}$ be the submodule of $\MC_{*,\alpha+\beta}(\G;\kk)_{A\to C}$ spanned by the chains that are not $\alpha$-special.
\end{definition}

\begin{lemma}\label{lem:general-nonspecial}
$N^\alpha_{*,\alpha+\beta}(\G;\kk)_{A\to C}$ is a subcomplex.
\end{lemma}

\begin{proof}
Consider a chain $\bD\in N^\alpha_{*,\alpha+\beta}(\G;\kk)_{A\to C}$ with $D_i$ a smooth node, that is
\[
d(D_{i-1},D_{i+1})=d(D_{i-1},D_i)+d(D_i,D_{i+1}).
\]
Then we must have the same relation for separating vectors
\[
\vec{d}(D_{i-1},D_{i+1})=\vec{d}(D_{i-1},D_i)+\vec{d}(D_i,D_{i+1}).
\]
Then the differential deleting $D_i$ preserves the total crossing vector $\alpha+\beta$ and cannot create a new occurrence of $\alpha$ in the cumulative crossing vectors.
\end{proof}

Denote the quotient complex by the following
\[
Sp^\alpha_{*,\alpha+\beta}(\G;\kk)_{A\to C}=\MC_{*,\alpha+\beta}(\G;\kk)_{A\to C}/N^\alpha_{*,\alpha+\beta}(\G;\kk)_{A\to C}.
\]

\begin{definition}
If $\bD=(D_0,\ldots,D_k)\in\MC_{k,\alpha+\beta}(\G;\kk)_{A\to C}$ is $\alpha$-special with $\lambda_i(\bD)=\alpha$, we call $D_i$ the \emph{special node} of $\bD$, and define the \emph{special prefix} and \emph{special suffix} of $\bD$ by
\[
\bD^{\text{pref}}=(D_0,\ldots,D_i),\quad \bD^{\text{suff}}=(D_i,\ldots,D_k).
\]
\end{definition}

\begin{remark}
If $\bD$ is $\alpha$-special, we must have its special node $D_i=B$ by parity. Then
\[
\bD^{\mathrm{pref}}\in\MC_{i,\alpha}(\G;\kk)_{A\to B},\quad \bD^{\mathrm{suff}}\in\MC_{k-i,\beta}(\G;\kk)_{B\to C}.
\]
\end{remark}

Splitting an $\alpha$-special chain into its special prefix and special suffix gives the following generalization of Proposition 5.4 of \cite{Koizumi2026a}.

\begin{proposition}\label{split-quotient}
There is a natural isomorphism of chain complexes
\[
Sp^\alpha_{*,\alpha+\beta}(\G;\kk)_{A\to C}\cong\MC_{*,\alpha}(\G;\kk)_{A\to B}\otimes_\kk\MC_{*,\beta}(\G;\kk)_{B\to C}.
\]
Equivalently, there is a chain map
\[
q_{\alpha,\beta,C}:\MC_{*,\alpha+\beta}(\G;\kk)_{A\to C}\to\MC_{*,\alpha}(\G;\kk)_{A\to B}\otimes_\kk\MC_{*,\beta}(\G;\kk)_{B\to C}
\]
which sends a special chain $\bD$ to $\bD^{\mathrm{pref}}\otimes\bD^{\mathrm{suff}}$ and a nonspecial chain to zero.
\end{proposition}

\begin{proof}
Concatenation and splitting are mutually inverse bijections on the preferred bases and hence provide the desired isomorphism of chain groups. To see these are also chain maps, we consider the boundary term of an $\alpha$-special chain $\bD$ deleting the special node $D_i$, it is either zero or nonspecial and hence vanishes in the quotient. Thus
\[
\partial\bD\equiv\sum_{j=1}^{i-1}(-1)^j\partial_j\bD+(-1)^i\sum_{j=1}^{k-i-1}(-1)^j\partial_{i+j}\bD\quad \mod N_{*,\alpha+\beta}^{\alpha}(\G;\kk)_{A\to C}.
\]
This is exactly the tensor product differential with the first tensor factor in degree $i$.
\end{proof}

\begin{proposition}\label{cup-dual-split}
Under the canonical identification
\[
\MC^{*,\alpha}(\G;\kk)_{A\to B}\otimes_{\kk}\MC^{*,\beta}(\G;\kk)_{B\to C}\cong\left(\MC_{*,\alpha}(\G;\kk)_{A\to B}\otimes_{\kk}\MC_{*,\beta}(\G;\kk)_{B\to C}\right)^\vee,
\]
the dual of the split map $q_{\alpha,\beta,C}$ is given by the magnitude cup product
\[
q_{\alpha,\beta,C}^{\vee}:\MC^{*,\alpha}(\G;\kk)_{A\to B}\otimes_{\kk}\MC^{*,\beta}(\G;\kk)_{B\to C}\xrightarrow{\smile}\MC^{*,\alpha+\beta}(\G;\kk)_{A\to C}.
\]
\end{proposition}

\begin{proof}
Take $\varphi\in\MC^{p,\alpha}(\G;\kk)_{A\to B}, \psi\in\MC^{q,\beta}(\G;\kk)_{B\to C}$ and $\bD\in\MC_{p+q,\alpha+\beta}(\G(\M);\kk)_{A\to C}$. If $\lambda_p(\bD)=\alpha$, then $D_p=B$ and
\[
q_{\alpha,\beta,C}(\bD)=(D_0,\ldots,D_p)\otimes(D_p,\ldots,D_{p+q}).
\]
Therefore, by \eqref{cup-chain},
\[
q_{\alpha,\beta,C}^\vee (\varphi\otimes\psi)(\bD)=\varphi(D_0,\ldots,D_p)\psi(D_p,\ldots,D_{p+q})=(\varphi\smile\psi)(\bD).
\]
If $\lambda_p(\bD)\neq\alpha$, then the degree $(p,q)$ component of $q_{\alpha,\beta,C}(\bD)$ is zero. The above equation remains true.
\end{proof}

\subsection{Periodicity of magnitude cohomology}

Koizumi proved the following periodicity theorem.

\begin{theorem}[{\cite[Theorem 3.2]{Koizumi2026a}}]
Let $\alpha\in\NN^E$ satisfy $U_1(\alpha)=E$ (that is $\alpha\geq \one_E$). For a tope $B\in\T(\M)$, we have
\[
\MH_{k,\alpha}(\G;\ZZ)_{\bullet\to B}\cong \MH_{k-\rk M,\alpha-\one_E}(\G;\ZZ)_{\bullet\to B}.    
\]
\end{theorem}

The goal of this subsection is to prove the dual of the periodicity theorem in mod $2$ cohomology is given by cup product. To this end, we first need a technical lemma.

\begin{lemma}\label{chain-localization-special}
Let $(\alpha,C)$ be tope-admissible and put
\[
F=U_1(\alpha),\quad X=C\setminus F,\quad A=\eps_\alpha C,\quad B=\eps_{\alpha-\one_F}C.
\]
Write
\[
\bar\alpha=\alpha|_F,\quad \bar A=A|_F,\quad \bar B=B|_F,\quad \bar C=C|_F,
\]
and let $\one^F\in\NN^F$ denote the all-ones vector on $F$. Then restriction to $F$ induces an isometric graph isomorphism
\begin{equation}\label{facial-localization-graph}
\rho_X:\G_X\xrightarrow{\cong}\G(\M|F), \quad D\longmapsto D|_F.
\end{equation}
For every $\delta\in\NN^E$ supported on $F$, it consequently induces a chain isomorphism
\begin{equation}\label{facial-localization-chain}
\rho_\delta:\MC_{*,\delta}(\G;\kk)_{D\to D'}\xrightarrow{\cong}\MC_{*,\delta|_F}(\G(\M|F);\kk)_{D|_F\to D'|_F}
\end{equation}
whenever $D,D'\in\T_X$. Under these identifications, we have
\begin{equation}\label{facial-localization-nonspecial}
\rho_\alpha\left(N^{\one_F}_{*,\alpha}(\G;\kk)_{A\to C}\right)=N^{\one^F}_{*,\bar\alpha}(\G(\M|F);\kk)_{\bar A\to\bar C}.
\end{equation}
Moreover, the following diagram commutes
\begin{equation}\label{facial-localization-split}
\xymatrix@C=3.2em@R=4.2em{
\MC_{*,\alpha}(\G;\kk)_{A\to C}
  \ar[r]^-{\scriptstyle q_{\one_F,\alpha-\one_F,C}}
  \ar[d]_-{\scriptstyle \rho_\alpha}^{\cong}
&
{\begin{array}{c}
\MC_{*,\one_F}(\G;\kk)_{A\to B}\\[-2pt]
{}\otimes
\MC_{*,\alpha-\one_F}(\G;\kk)_{B\to C}
\end{array}}
  \ar[d]^-{\scriptstyle
    \rho_{\one_F}\otimes\rho_{\alpha-\one_F}}_{\cong}
\\
\MC_{*,\bar\alpha}(\G(\M|F);\kk)_{\bar A\to\bar C}
  \ar[r]^-{\scriptstyle
    q_{\one^F,\bar\alpha-\one^F,\bar C}}
&
{\begin{array}{c}
\MC_{*,\one^F}(\G(\M|F);\kk)_{\bar A\to\bar B}\\[-2pt]
{}\otimes
\MC_{*,\bar\alpha-\one^F}
  (\G(\M|F);\kk)_{\bar B\to\bar C}
\end{array}}
}
\end{equation}
\end{lemma}

\begin{proof}
Since $(\alpha,C)$ is tope-admissible, $F$ is a flat and $X=C\setminus F$ is a covector with zero set $z(X)=F$. Recall that
\[
\T_X=\{D\in\T(\M)\mid X\leq D\}.
\]
If $D\in\T_X$, then $D|_F$ is a tope of $\M|F$. Conversely, let $\bar D\in\T(\M|F)$. By the definition of restriction, there is a covector $Y\in\Lcov(\M)$ such that $Y|_F=\bar D$. The Tits product \eqref{titsproduct} gives $X\circ Y\in\Lcov(\M)$. Since $X$ is nonzero outside $F$ and $\bar D$ is nonzero on $F$, the sign vector $X\circ Y$ has no zero entries, meaning that it is a tope in $\T_X$, and $(X\circ Y)|_F=\bar D$. This proves that $\rho_X$ is a bijection on vertices.

Any two topes in $\T_X$ agree outside $F$, and hence
\[
 \Sep_{\M}(D,D')=\Sep_{\M|F}(D|_F,D'|_F).
\]
Because $M$ is simple, so is $M|F$, and the tope metrics on both sides are given by the cardinality of the separating set. Thus $\rho_X$ is an isometry and hence a graph isomorphism, proving \eqref{facial-localization-graph}.

Now let $\delta$ be supported on $F$. Every chain of crossing vector $\delta$ whose endpoints lie in $\T_X$ crosses no coordinate outside $F$, so all of its nodes lie in $\T_X$. The isometry $\rho_X$ preserves the crossing vector, the magnitude length, and the condition that an internal node is smooth. Applying $\rho_X$ nodewise therefore gives the chain isomorphism \eqref{facial-localization-chain}.

For a chain $\bD=(D_0,\ldots,D_k)$ of crossing vector $\alpha$, restriction also preserves every cumulative crossing vector $\lambda_i(\rho_X\bD)=\lambda_i(\bD)|_F$. Since all coordinates outside $F$ are zero, this implies
\[
\lambda_i(\bD)=\one_F \iff \lambda_i(\rho_X\bD)=\one^F.
\]
Thus a chain is $\one_F$-special in the ambient graph if and only if its localized chain is $\one^F$-special. This proves \eqref{facial-localization-nonspecial}. Finally, restriction neither changes the special index nor the special prefix and the special suffix, so it intertwines the split maps, giving \eqref{facial-localization-split}.
\end{proof}

For a nonempty flat $F\in L(M)$ and a tope $B\in\T(\M)$ satisfying $B\setminus F\in\Lcov(\M)$, that is $(\one_F,B)$ is tope-admissible, define
\begin{equation}\label{single-flat}
x_{F,B}=u_{\one_F,B}\in\MH^{\rk F,\one_F}(\G;\Ftwo)_{\eps_FB\to B}.    
\end{equation}
We set $x_{\varnothing,B}=e_B$. We call such classes \emph{single-flat classes}. Now we can prove the cohomology periodicity theorem.

\begin{theorem}\label{periodicity-product}
Let $(\alpha,C)$ be tope-admissible. Put $F=U_1(\alpha)$ and $B=\eps_{\alpha-\one_F}C$. Then
\[
x_{F,B}\smile-:\MH^{\kappa_M(\alpha-\one_F),\alpha-\one_F}(\G;\Ftwo)_{B\to C}\xrightarrow{\cong}\MH^{\kappa_M(\alpha),\alpha}(\G;\Ftwo)_{\eps_\alpha C\to C}
\]
is an isomorphism, and
\[
u_{\alpha,C}=x_{F,B}\smile u_{\alpha-\one_F,C}.
\]
\end{theorem}

\begin{proof}
If $\alpha=\zero$, then $F=\varnothing$, $B=C$, and the assertion is the identity map given by the idempotent $e_C$. Next assume that $F$ is nonempty. First note that $(\alpha-\one_F,C)$ is tope-admissible, because
\[
U_j(\alpha-\one_F)=U_{j+1}(\alpha),\quad j\geq1.
\]
Moreover, $B$ differs from $C$ only on $F$, so $B\setminus F=C\setminus F$ is a covector. Hence the class $x_{F,B}$ and the source of the displayed map are both defined.

Use the notation of \Cref{chain-localization-special}. The restricted crossing vector $\bar\alpha=\alpha|_F$ has full support on $F$, so $\bar\alpha\geq\one^F$. The oriented matroid $\M|F$ is realized by the subarrangement indexed by $F$ and has rank $\rk F$. Moreover, $A=\eps_F B$, so $\bar B=-\bar A$ in $\T(\M|F)$. Koizumi's acyclicity theorem for $\one$-nonspecial chains \cite[Section 5]{Koizumi2026a} therefore gives
\[
H_*\left(N^{\one^F}_{*,\bar\alpha}(\G(\M|F);\Ftwo)_{\bar A\to\bar C}\right)=0.
\]
Indeed, the nonspecial complex decomposes as the direct sum of its fixed-endpoint subcomplexes, and the differential preserves the endpoints. By \eqref{facial-localization-nonspecial}, the ambient nonspecial subcomplex $N^{\one_F}_{*,\alpha}(\G;\Ftwo)_{A\to C}$ is also acyclic. Hence the split map
\[
q_{\one_F,\alpha-\one_F,C}:\MC_{*,\alpha}(\G;\Ftwo)_{A\to C}\to\MC_{*,\one_F}(\G;\Ftwo)_{A\to B}\otimes_{\Ftwo}\MC_{*,\alpha-\one_F}(\G;\Ftwo)_{B\to C}
\]
is a quasi-isomorphism.

All chain groups are finite-dimensional over $\Ftwo$, so dualizing this quasi-isomorphism gives a quasi-isomorphism of cochain complexes. By \Cref{cup-dual-split}, that dual map is the endpoint-restricted cup product. The first tensor factor has cohomology concentrated in degree $\rk F$ and is one-dimensional there, with its unique nonzero class equal to $x_{F,B}$ by \Cref{Koizumi-basis}. The K\"unneth theorem over $\Ftwo$ therefore shows that cup product with $x_{F,B}$ induces an isomorphism from the second factor to the target.

Finally,
\[
\kappa_M(\alpha)=\rk F+\kappa_M(\alpha-\one_F),
\]
because the level sets of $\alpha-\one_F$ are $U_{j+1}(\alpha)$. The source and target in the asserted bidegrees are both one-dimensional. The isomorphism sends a nonzero class to a nonzero class, and over $\Ftwo$ these are the canonical classes. Hence
\[
u_{\alpha,C}=x_{F,B}\smile u_{\alpha-\one_F,C}.
\]
\end{proof}

Iterating gives a canonical standard monomial.

\begin{corollary}\label{standard-factorization}
Let $(\alpha,C)$ be tope-admissible and write its nonempty level flats, with repetitions, as
\[
F_1\supseteq F_2\supseteq\cdots\supseteq F_m,\quad F_j=U_j(\alpha).
\]
Set $C_m=C$ and recursively $C_{j-1}=\eps_{F_j}C_j$. Then every $x_{F_j,C_j}$ is defined and
\begin{equation}\label{standard-factorization-eq}
u_{\alpha,C}=x_{F_1,C_1}\smile x_{F_2,C_2}\smile\cdots \smile x_{F_m,C_m}.
\end{equation}
\end{corollary}

\begin{proof}
Since $F_{j+1}\subseteq F_j$, the tope $C_j$ differs from $C$ only on $F_j$. Hence $C_j\setminus F_j=C\setminus F_j$ is a covector. Repeated applications of \Cref{periodicity-product} give the result.
\end{proof}

\section{Integral bases for antipodal intervals}\label{integral-basis}

This section proves the topological tools used in next section. The result is integral, although the rest of the paper only needs its reduction modulo $2$.

\subsection{A simultaneous deletion lemma}

We begin with a useful elementary observation about finite posets.

\begin{lemma}\label{simultaneous-deletion}
Let $P$ be a finite poset and let $S\subseteq P$.
\begin{enumerate}[label=\textup{(\roman*)}]
\item If $P_{>x}$ is contractible for every $x\in S$, then the inclusion $P\setminus S\hookrightarrow P$ is a homotopy equivalence.
\item If $P_{<x}$ is contractible for every $x\in S$, then the inclusion $P\setminus S\hookrightarrow P$ is a homotopy equivalence.
\end{enumerate}
\end{lemma}

\begin{proof}
It is enough to prove (i), since (ii) is its order dual. We induct on $|S|$. Choose a maximal element $x$ of $S$. Its strict upper interval $P_{>x}$ contains no other element of $S$, so deleting $x$ is a homotopy equivalence by the standard one-point poset deletion lemma (see \cite[Lemma 3.1]{Barmak2011}). For $y\in S\setminus\{x\}$, the upper interval of $y$ after deleting $x$ is either unchanged or is $P_{>y}\setminus\{x\}$. In the latter case, the strict upper interval of $x$ inside $P_{>y}$ is $P_{>x}$, which is contractible. Hence
\[
P_{>y}\setminus\{x\}\simeq P_{>y}
\]
by the one-point deletion lemma. The remaining upper intervals are therefore contractible, and induction applies.
\end{proof}

\subsection{Boolean embedding}

Recall the discussion in \Cref{initial-tope-graph}. Let $\M$ have rank $r$, let
\[
\F:\quad \varnothing=F_0\subsetneq F_1\subsetneq\cdots\subsetneq F_r=E
\]
be a complete flag of flats, and put $S_i=F_i\setminus F_{i-1}$. Fix a tope $A$ incident to the flag, equivalently $A\in\T(\M_{\mathcal F})$. Then $A_J=\eps_{\bigcup_{i\in J}S_i}A$ is a tope for $J\subseteq[r]$ with
\[
\Sep_\M(A,A_J)=\bigcup_{i\in J}S_i.
\]
Therefore $J\mapsto A_J$ defines a poset embedding
\begin{equation}\label{boolean-embedding}
s:\Bl_r\hookrightarrow [A,-A]_{\G}
\end{equation}
from the Boolean lattice. It sends $\varnothing$ to $A$ and $[r]$ to $-A$. Abbreviate
\[
P=(A,-A)_{\G},\quad \overline{\Bl}_r=\Bl_r\setminus\{\varnothing,[r]\}.
\]
The restriction of $s$ gives an embedding $s:\overline{\Bl}_r\hookrightarrow P$.

Choose an element $e_i\in S_i$ and put
\[
\bb=\{e_1,\ldots,e_r\}.
\]
Since $F_{i-1}$ is a flat and $e_i\notin F_{i-1}$, $\bb$ is a basis of $M$. Restricting sign vectors to $\bb$ gives an order-preserving map
\[
t:[A,-A]_{\G}\to \Bl_r,\quad t(C)=\{i\in[r]\mid C_{e_i}\neq A_{e_i}\}.
\]
In particular, $\rho(A_J)=J$. We single out the topes in $P$ that are mapped by $t$ to $\varnothing$ and $[r]$ respectively,
\[
L_+=\{C\in P\mid t(C)=\varnothing\},\quad L_-=\{C\in P\mid t(C)=[r]\},
\]
and put
\[
P^{\mathrm{mid}}=P\setminus(L_+\cup L_-).
\]

\begin{lemma}\label{basis-separation}
Let $C<D$ in $[A,-A]_{\G}$ and suppose $\bb\subseteq\Sep(C,D)$. Then the open interval $(C,D)_{\G}$ is either equal to $P$ or is contractible.
\end{lemma}

\begin{proof}
If $[C,D]_{\G}$ is facial with respect to $X$, then $z(X)=\Sep(C,D)$. A flat containing the basis $\bb$ is all of $E$, so $X=\zero$ and $[C,D]_{\G}=[A,-A]_{\G}$. This forces $C=A$ and $D=-A$. Every other interval is contractible by \Cref{thm:EW}.
\end{proof}

\begin{proposition}\label{mid-homotopy}
The inclusion $P^{\mathrm{mid}}\hookrightarrow P$ is a homotopy equivalence.
\end{proposition}

\begin{proof}
For $C\in L_+$, the basis is contained in $\Sep(C,-A)$. Since $A\not\in L_+$, $C\neq A$ and \Cref{basis-separation} shows that $P_{>C}=(C,-A)_{\G}$ is contractible. By \Cref{simultaneous-deletion},
\[
R:=P\setminus L_+\simeq P.
\]

Now let $D\in L_-$. The basis is contained in $\Sep(A,D)$, and $D\neq -A$, so $(A,D)_{\G}$ is contractible. Its intersection with $L_+$ can be deleted without changing homotopy type. In fact, for $C\in L_+\cap(A,D)_{\G}$, the interval $(C,D)_{\G}$ is contractible by \Cref{basis-separation}. Therefore
\[
R_{<D}:=(A,D)_{\G}\setminus L_+
\]
is contractible. Another application of \Cref{simultaneous-deletion} gives
\[
P^{\mathrm{mid}}=R\setminus L_-\simeq R\simeq P.
\]
\end{proof}

Denote the restriction of $t$ to $P^{\mathrm{mid}}$ by $\overline{t}:P^{\mathrm{mid}}\to\overline{\Bl}_r$. Note that $s(\overline{\Bl}_r)\subseteq P^{\mathrm{mid}}$. Then
\[
\overline{t}\circ s=\id_{\overline{\Bl}_r}.
\]

\begin{theorem}\label{basis-thm}
The embedding $s:\overline{\Bl}_r\hookrightarrow P$ induces an isomorphism
\[
s_*:\widetilde H_{r-2}(\overline{\Bl}_r;\ZZ)\xrightarrow{\cong}\widetilde H_{r-2}(P;\ZZ).
\]
\end{theorem}

\begin{proof}
For $r=1$, both posets are empty and the assertion is immediate. Assume $r\ge2$. The result follows from $\overline{t}\circ s=\id_{\overline{\Bl}_r}$ and \Cref{mid-homotopy}.
\end{proof}

\subsection{Integral basis}

For a permutation $\pi\in \mathfrak{S}_r$, put
\[
A_{\pi,j}=\eps_{S_{\pi(1)}\cup\cdots\cup S_{\pi(j)}}A,\quad 0\leq j\leq r.
\]
Define the integral proper chain
\begin{equation}\label{integral-cycle}
\omega_{\F,A}=\sum_{\pi\in \mathfrak{S}_r}\sgn(\pi)(A_{\pi,0},A_{\pi,1},\ldots,A_{\pi,r})\in\MC_{r,\one_E}(\G;\ZZ).
\end{equation}

\begin{corollary}\label{antipodal-basis}
The chain $\omega_{\mathcal F,A}$ is a cycle and its class is a generator of
\[
\MH_{r,\one_E}(\G(\M);\ZZ)_{A\to -A}\cong\ZZ.
\]
\end{corollary}

\begin{proof}
For a geodesic endpoint pair, magnitude chains identify with the augmented simplicial chains of the open interval, shifted by two \cite[Proposition~2.3]{Koizumi2026a}. Under this identification, \eqref{integral-cycle} maps, up to one global sign, to the barycentric fundamental cycle of $\Delta\overline{\Bl}_r$. Then \Cref{basis-thm} proves the claim.
\end{proof}

\section{The local modular relation}\label{local-modular}

We first deal with a base case. Let $F,G$ be flats. We call $(F,G)$ a \emph{closed modular pair} if $F\cup G$ is a flat and
\begin{equation}\label{modular-def}
\rk F+\rk G=\rk F\cup G+\rk F\cap G.
\end{equation}

\subsection{Direct sum decomposition}

\begin{lemma}\label{modular-direct-sum}
If $(F,G)$ is a closed modular pair, then
\begin{equation}\label{modular-sum-eq}
(\M|F\cup G)/(F\cap G)\cong((\M|F)/(F\cap G))\oplus((\M|G)/(F\cap G)).
\end{equation}
The two factors have ground sets $F\setminus G$ and $G\setminus F$.
\end{lemma}

\begin{proof}
The ground sets are disjoint and cover $(F\cup G)\setminus (F\cap G)$. In the contraction $N=(M|F\cup G)/(F\cap G)$,
\[
\rk_N(F\setminus G)=\rk F-\rk F\cap G, \quad \rk_N(G\setminus F)=\rk G-\rk F\cap G.
\]
Equation \eqref{modular-def} says that their ranks add to $\rk N$. This is the direct sum criterion for a matroid, and the oriented direct sum structure is inherited from the minors.
\end{proof}

The direct sum makes the tope-incidence conditions transparent.

\begin{lemma}\label{incidence-refinement}
Assume that $(F,G)$ is a closed modular pair and $C\in\T(\M)$ such that $C\setminus (F\cap G)$ and $C\setminus (F\cup G)$ are covectors. Then $C\setminus F,C\setminus G$ and $(\eps_G C)\setminus F$ are covectors. In particular, all single-flat classes \eqref{single-flat} occurring in
\[
x_{F,\eps_G C}\smile x_{G,C}\quad\text{and}\quad x_{F\cup G,\eps_{F\cap G}C}\smile x_{F\cap G,C}
\]
are defined.
\end{lemma}

\begin{proof}
The assumptions say that $C$ is a tope of the initial oriented matroid for the flag $F\cap G\subseteq F\cup G$ (see \Cref{initial-topes}). Its middle summand is $(\M|F\cup G)/(F\cap G)$, which splits as in \eqref{modular-sum-eq}. Inserting $F$ or $G$ in the flag merely separates the two direct summands and does not change the initial oriented matroid. Hence $C$ is incident to $F$ and $G$. Reversing all signs on $G=(F\cap G)\cup(G\setminus F)$ reverses a union of direct summands, so $\eps_G C$ is again a tope of the same initial oriented matroid and is incident to $F$.
\end{proof}

Under the assumptions of \Cref{incidence-refinement}, put
\[
A=\eps_F\eps_G C,\quad B=\eps_G C,\quad D=\eps_{F\cap G}C,\quad \gamma=\one_F+\one_G=\one_{F\cup G}+\one_{F\cap G}.
\]
We use these notations throughout this section. By \eqref{modular-def},
\begin{equation}\label{modular-degree}
\kappa_M(\gamma)=\rk F\cup G+\rk F\cap G=\rk F+\rk G.
\end{equation}

\subsection{Compatible complete flags}

Write
\[
a=\rk F\cap G, \quad b=\rk F-\rk F\cap G, \quad c=\rk G-\rk F\cap G.
\]
In the summand $\M|F\cap G$, let us choose a complete flag of flats incident to $A|_{F\cap G}$
\[
\varnothing\subsetneq S_1 \subsetneq S_1\sqcup S_2 \subsetneq\cdots\subsetneq S_1\sqcup S_2\sqcup\cdots \sqcup S_a=F\cap G,
\]
where $\varnothing\neq S_i\subsetneq F\cap G$. Similarly, we choose a complete flag of flats
\[
\varnothing\subsetneq T_1 \subsetneq T_1\sqcup T_2 \subsetneq\cdots\subsetneq T_1\sqcup T_2\sqcup\cdots \sqcup T_b=F\setminus G
\]
in $(\M|F)/(F\cap G)$ incident to $A|_{F\setminus G}$, and a complete flag of flats
\[
\varnothing\subsetneq V_1 \subsetneq V_1\sqcup V_2 \subsetneq\cdots\subsetneq V_1\sqcup V_2\sqcup\cdots \sqcup V_c=G\setminus F
\]
in $(\M|G)/(F\cap G)$ incident to $A|_{G\setminus F}$. By \Cref{modular-direct-sum,incidence-refinement}, these flags determine compatible Boolean embeddings \eqref{boolean-embedding} for the flats $F\cap G,F,G,F\cup G$. Every sign vector obtained from $A$ by independently reversing the blocks $S_i,T_j,U_k$ is therefore a tope.

For each block $S_i$ introduce two labeled blocks with ordering
\[
S_i^-<S_i^+.
\]
Every $T_j$ and $U_k$ remain one block, with no additional comparabilities. Let $\mathscr{E}$ be the resulting poset of blocks. The cardinality of $\mathscr{E}$ is $d=2a+b+c=\rk F+\rk G$.

For a linear extension $\pi=(e_1,\ldots,e_d)$ of $\mathscr{E}$, start at $A$ and reverse the block labeled by $e_1$, then $e_2$, and so on. Denote the resulting magnitude chain by
\[
\bA_\pi=(A=A_0^\pi,A_1^\pi,\ldots,A_d^\pi=C).
\]
Every $S_i$ is crossed twice, while every $T_j$ and $U_k$ is crossed once, so $\bA_\pi$ has crossing vector $\gamma$. Define
\[
\Omega_{F,G;C}=\sum_{\pi\in\operatorname{LinExt}(\mathscr{E})}\bA_\pi\in\MC_{d,\gamma}(\G;\Ftwo)_{A\to C}.
\]

\begin{lemma}
The chain $\Omega_{F,G;C}$ is a cycle.
\end{lemma}

\begin{proof}
Consider a possible boundary term of $\bA_\pi$ obtained by deleting an intermediate node $A_i^\pi$. If $e_i,e_{i+1}$ are incomparable, these two steps reverse distinct blocks, their crossing supports are disjoint and the node is smooth. The same shortened chain occurs for the linear extension $(i,i+1)\pi$, so the two terms cancel over $\Ftwo$. If $\{e_i,e_{i+1}\}$ are $\{S_i^-,S_i^+\}$, they reverse the same block twice.  Then $A_{i-1}^\pi=A_{i+1}^\pi$, so $A_i^\pi$ is not smooth and contributes no boundary. These are all cases.
\end{proof}

\subsection{Splitting the multishuffle}

Let
\[
q_{\one_F,\one_G,C}:\MC_{*,\gamma}(\G;\Ftwo)_{A\to C}\to\MC_{*,\one_F}(\G;\Ftwo)_{A\to B}\otimes_{\Ftwo}\MC_{*,\one_G}(\G;\Ftwo)_{B\to C}
\]
be the split map of \Cref{split-quotient} with $\kk=\Ftwo$.

A linear extension reaches cumulative crossing vector $\one_F$ exactly when all blocks $S_i^-$ and $T_j$ have occurred, while no $S_i^+$ or $U_k$ has occurred. Hence the $\one_F$-special linear extensions are precisely the concatenations
\[
\bigl(\text{a permutation of }S_i^-,T_j\bigr)\bigl(\text{a permutation of }S_i^+,U_k\bigr).
\]
Let $\overline{\omega}_F$ and $\overline{\omega}_G$ denote the mod-$2$ cycles \eqref{integral-cycle} for the compatible complete flags of $F$ and $G$.  We obtain
\[
q_{\one_F,\one_G,C}(\Omega_{F,G;C})=\overline{\omega}_F\otimes\overline{\omega}_G.
\]

By \Cref{antipodal-basis}, both tensor factors represent the unique nonzero classes in their antipodal mod $2$ magnitude homology groups. Therefore
\begin{equation}\label{split-nonzero}
H(q_{\one_F,\one_G,C})[\Omega_{F,G;C}]\neq0.
\end{equation}

\begin{theorem}\label{local-modular-thm}
Let $F,G$ be a closed modular pair and let $C\in\T(\M)$ satisfy $C\setminus (F\cap G),C\setminus (F\cup G)\in\Lcov(\M)$. Then
\[
x_{F,\eps_G C}\smile x_{G,C}=x_{F\cup G,\eps_{F\cap G}C}\smile x_{F\cap G,C}=u_{\one_F+\one_G,C}.
\]
\end{theorem}

\begin{proof}
By \Cref{Koizumi-basis,modular-degree}, the source of $H(q_{\one_F,\one_G,C})$ and its target are both one-dimensional and concentrated in degree $d$. Equation \eqref{split-nonzero} therefore shows that $H(q_{\one_F,\one_G,C})$ is an isomorphism. By \Cref{cup-dual-split}, the product $x_{F,\eps_G C}\smile x_{G,C}$ is nonzero, hence equals the unique nonzero target class $u_{\gamma,C}$. The equality
\[
x_{F\cup G,\eps_{F\cap G}C}\smile x_{F\cap G,C}=u_{\gamma,C}
\]
is the standard factorization \Cref{standard-factorization} for the level-flat chain $F\cup G\supseteq F\cap G$.
\end{proof}

\section{The product formula}
\label{main-result}

Now we may deal with the general case.

\subsection{Sorting flat words}

A \emph{flat word} is a finite sequence $\bF=(F_1,\ldots,F_m)$ of nonempty flats. Its degree and multiplicity vector are
\[
d(\bF)=\sum_{i=1}^m\rk_MF_i\in\NN, \quad w(\bF)=\sum_{i=1}^m\one_{F_i}\in\NN^E.
\]
For a pair $(\bF,C)$ of a flat word $\bF$ and a tope $C\in\T(\M)$, define $C_m=C$ and recursively
\[
C_{j-1}=\eps_{F_j}C_j\in\{+,-\}^E, \quad j=m,m-1,\ldots,1.
\]
Note that $C_j$ are not necessarily topes. Whenever every single-flat class (\Cref{single-flat}) $x_{F_j,C_j}\in\MH^{\rk F_j,\one_{F_j}}(\G(\M);\Ftwo)_{C_{j-1}\to C_j}$ is defined, put
\begin{equation}\label{flat-word-monomial}
x_{\bF,C}=x_{F_1,C_1}\smile\cdots \smile x_{F_m,C_m}\in\MH^{d(\bF),w(\bF)}(\G(\M);\Ftwo)_{C_0\to C_m}.
\end{equation}

\begin{definition}
The pair $(\bF,C)$ is \emph{clean} if $(w(\bF),C)$ is tope-admissible and $d(\bF)=\kappa_M(w(\bF))$.
\end{definition}

By \Cref{u-class}, when $(\bF,C)$ is clean, the class $u_{w(\bF),C}\in\MH^{d(\bF),w(\bF)}(\G(\M);\Ftwo)_{\eps_{w(\bF)}C\to C}$ is defined. Moreover, we have the following.

\begin{proposition}\label{modular-straightening}
If $(\bF,C)$ is clean, then every factor in \eqref{flat-word-monomial} is defined and
\[
x_{\bF,C}=u_{w(\bF),C}.
\]
More precisely, every elementary sorting move \eqref{sorting-move} on $\bF$ is realized by the local modular relation.
\end{proposition}

\begin{proof}
Put $w=w(\bF)$, let $\D$ be the family generated by $F_1,\ldots,F_m,\varnothing,E$ under unions and intersections (\Cref{tight-flat-lattice}), and set $\N=\M_w$. Since $(\bF,C)$ is clean, $w$ is $M$-admissible and $\kappa_M(w)=\sum_{j=1}^m\rk F_j$. By \Cref{admissible-initial-loopless}, the initial matroid $M_w$ is loopless and $C\in\T(\N)$. Then \Cref{tight-flat-lattice} guarantees that every member of $\D$ is a flat of $M$ and a separator of $M_w$, and every pair of members of $\D$ is closed modular.

We first check that all single-flat classes occurring during sorting are defined. Let
\[
\bH=(H_1,\ldots,H_m)
\]
be any word obtained from $\bF$ by sorting moves \eqref{sorting-move}, keeping an empty set temporarily when it occurs. Every $H_j$ belongs to $\D$. With terminal tope $D_m=C$, define $D_{j-1}=\eps_{H_j}D_j$. Equivalently,
\[
D_j=\eps_{H_{j+1}\mathbin\triangle\cdots\mathbin\triangle H_m}C,
\]
where $\triangle$ denotes symmetric difference. Separators of $M_w$ are unions of connected components and therefore form a Boolean algebra under union, intersection, complement, and symmetric difference. It follows that each $H_{j+1}\mathbin\triangle\cdots\mathbin\triangle H_m$ is again a separator of $M_w$. Since $C\in\T(\N)$, reversing such a separator gives another tope of $\N$, so $D_j\in\T(\N)$ for every $j$. Applying \Cref{tight-tope-incidence} to $H_j\in\D$ gives
\[
D_j\setminus H_j\in\Lcov(\M).
\]
Thus every class $x_{H_j,D_j}$ appearing at every stage of the sorting process is defined.

Now suppose an adjacent sorting-eligible pair in the current word is $(F,G)$, and let $D$ be the terminal tope of this two-factor subword. By \Cref{tight-flat-lattice}, $(F,G)$ is a closed modular pair and $F\cap G,F\cup G\in\D$. By \Cref{tight-tope-incidence}, $D\setminus (F\cap G),D\setminus(F\cup G)\in\Lcov(\M)$. Therefore \Cref{local-modular-thm} applies and gives
\begin{equation}\label{one-straightening-step}
x_{F,\eps_GD}\smile x_{G,D}=x_{F\cup G,\eps_{F\cap G}D}\smile x_{F\cap G,D}.
\end{equation}
This is exactly the algebraic realization of the sorting move $(F,G)\mapsto(F\cup G,F\cap G)$. It preserves the multiplicity vector $\one_F+\one_G=\one_{F\cup G}+\one_{F\cap G}$, and it preserves the total degree by the modular rank equality \eqref{modular-def}.

Repeatedly applying \eqref{one-straightening-step} and then \Cref{set-sorting}, the word terminates at
\[
U_1(w)\supseteq U_2(w)\supseteq\cdots\supseteq U_{m(w)}(w),
\]
with empty sets omitted. The product has not changed during this process. The terminal nested monomial is precisely the standard factorization of $u_{w,C}$ from \Cref{standard-factorization}. This proves $x_{\bF,C}=u_{w(\bF),C}$.
\end{proof}

\subsection{The product formula}

We can now determine every structure constant.

\begin{theorem}\label{product-formula}
Let $(\alpha,B)$ and $(\beta,C)$ be tope-admissible in $\M$. Then
\begin{equation}\label{product-eq}
u_{\alpha,B}\smile u_{\beta,C}=\begin{cases}
 u_{\alpha+\beta,C}, &\begin{array}{l}
       B=\eps_\beta C,\\
       (\alpha+\beta,C)\text{ is tope-admissible},\\
       \kappa_M(\alpha+\beta)=\kappa_M(\alpha)+\kappa_M(\beta),
 \end{array}\\[6mm]
 0,&\text{otherwise}.
 \end{cases}
\end{equation}
\end{theorem}

\begin{proof}
If $B\neq\eps_\beta C$, the endpoint summands do not compose, so the cup product is zero. Assume $B=\eps_\beta C$. The product has crossing vector $\alpha+\beta$ and cohomological degree $\kappa_M(\alpha)+\kappa_M(\beta)$. If $(\alpha+\beta,C)$ is not tope-admissible, then every cohomology group in the $(\alpha+\beta)$-crossing endpoint summand ending at $C$ is zero by \Cref{Koizumi-basis}. By \Cref{kappa-support}, if $\kappa_M(\alpha+\beta)<\kappa_M(\alpha)+\kappa_M(\beta)$, then \Cref{Koizumi-basis} shows that the target crossing summand vanishes in the cohomological degree of the product. These are the otherwise cases.

Next suppose the three displayed conditions are satisfied. Put $p=\max_{e\in E}\alpha_e$ and $q=\max_{e\in E}\beta_e$ and form flat words
\[
\bF=(F_1,\ldots,F_p),\quad F_i=U_i(\alpha), \quad \text{and} \quad \bG=(G_1,\ldots,G_q),\quad G_j=U_j(\beta).
\]
Also name sign vectors $A_p=B, A_{i-1}=\eps_{F_i}A_i$ and $D_q=C, D_{j-1}=\eps_{G_j}D_j$. By \Cref{layercake}, we have $\alpha=w(\bF)$ and $\beta=w(\bG)$. Moreover $(\bF,B)$ and $(\bG,C)$ are clean since $(\alpha,B),(\beta,C)$ are supposed to be tope-admissible and 
\begin{equation}\label{degree-kappa-eq}
d(\bF)=\sum_{i=1}^p\rk U_i(\alpha)=\kappa_M(\alpha),\quad d(\bG)=\sum_{j=1}^q\rk U_j(\beta)=\kappa_M(\beta),  
\end{equation}
Therefore
\begin{equation}\label{two-factors}
u_{\alpha,B}=x_{\bF,B}=x_{F_1,A_1}\smile\cdots\smile x_{F_p,A_p}, \quad u_{\beta,C}=x_{\bG,C}=x_{G_1,D_1}\smile\cdots\smile x_{G_q,D_q},
\end{equation}
are defined by \Cref{modular-straightening}. Since $D_0=\eps_\beta C=B=A_p$, these two monomials are composable.

Let $\bH=\bF\star\bG=(F_1,\ldots,F_p,G_1,\ldots,G_q)$ be the concatenation word. We claim $(\bH,C)$ is clean. In fact, $w(\bH)=\sum_{i=1}^p\one_{F_i}+\sum_{j=1}^q\one_{G_j}=\alpha+\beta$ and we already assumed $(\alpha+\beta,C)$ is tope-admissible; on the other hand
\[
d(\bH)=\sum_{i=1}^p\rk F_i+\sum_{j=1}^q\rk G_j=\kappa_M(\alpha)+\kappa_M(\beta)=\kappa_M(\alpha+\beta),
\]
where the second equality is \eqref{degree-kappa-eq} and the third is our assumption. Then \Cref{modular-straightening} again shows that
\[
x_{\bH,C}=x_{F_1,A_1}\smile\cdots\smile x_{F_p,A_p}\smile x_{G_1,D_1}\smile\cdots\smile x_{G_q,D_q}
\]
is defined and $u_{\alpha+\beta,C}=x_{\bH,C}$. Then combining with \Cref{two-factors}, we conclude
\[
u_{\alpha+\beta,C}=u_{\alpha,B}\smile u_{\beta,C}.
\]
\end{proof}

We leave the following problem as a conjecture.
\begin{conjecture}
The same product formula in \Cref{product-formula} holds over $\ZZ$.
\end{conjecture}

\subsection{Example: Rank-two arrangements}
\label{subsec:pencil-ring}

Let $\A$ be a central arrangement of $n\geq 3$ lines in $\R^2$. The underlying matroid is $U_{2,n}$ and the tope graph is the cycle $\G\cong C_{2n}$. The flats of $U_{2,n}$ are $\varnothing,\{i\}\ (i\in E),E$. For a chamber $C\in\T$, put
\[
\mathrm{bd}(C)=\{i\in E \mid C\setminus\{i\}\in\Lcov(\M)\}.
\]
These are the two lines bounding the chamber $C$.

\begin{proposition}
A vector $\alpha\in\NN^E$ is $U_{2,n}$-admissible if and only if
\[
\alpha=a\one_E+b\one_{\{i\}}
\]
for some $a,b\in\NN$ and $i\in E$. Moreover, $\kappa_M(\alpha)=2a+b$.

For $U_{2,n}$-admissible vector $\alpha$ as above, the pair $(\alpha,C)$ is tope-admissible if and only if $b=0$ or $i\in\mathrm{bd}(C)$.
\end{proposition}

\begin{proof}
Put $a=\min_{e\in E}\alpha_e$. Then $U_j(\alpha)=E$ for $1\leq j\leq a$. Every subsequent nonempty level set is a proper flat of $U_{2,n}$, hence a singleton. Since the level sets are nested, all these singletons are equal to one fixed $\{i\}$. This gives the asserted form of $\alpha$, and the converse is immediate. The formula for $\kappa_M$ follows because $E$ has rank $2$ and $\{i\}$ has rank $1$. Finally, incidence with $E$ is automatic, while incidence with $\{i\}$ is equivalent to $i\in\mathrm{bd}(C)$.
\end{proof}

Consequently, the magnitude cohomology ring has the canonical basis
\[
u_{a\one_E+b\one_{\{i\}},C}\quad \bigl( a\geq 0,\ b\geq 0,\ i\in\mathrm{bd}(C) \bigr).
\]
The class has cohomological degree $2a+b$ and ordinary length degrees $an+b$. Now put
\[
\alpha=a\one_E+b\one_{\{i\}}, \quad \beta=c\one_E+d\one_{\{j\}},
\]
and suppose that $u_{\alpha,B}$ and $u_{\beta,C}$ exist. \Cref{product-formula} gives
\[
u_{\alpha,B}u_{\beta,C}=\begin{cases}
 u_{\alpha+\beta,C},&\begin{array}{l}
 B=\eps_\beta C,\\
 b=0\text{ or }d=0\text{ or }i=j,
 \end{array}\\[4mm]
 0,&\text{otherwise}.
 \end{cases}
\]
Indeed, if $b,d>0$ and $i\neq j$, then one of the level sets of
$\alpha+\beta$ is $\{i,j\}$, which is not a flat of $U_{2,n}$.

\section*{Acknowledgments}
The author thanks Junnosuke Koizumi for valuable comments.

\printbibliography

@book {Anderson2025,
    AUTHOR = {Anderson, Laura},
     TITLE = {Oriented matroids},
    SERIES = {Cambridge Studies in Advanced Mathematics},
    VOLUME = {216},
 PUBLISHER = {Cambridge University Press, Cambridge},
      YEAR = {2025},
     PAGES = {xii+321},
      ISBN = {9-781-009-49411-3; [9781009494076]},
   MRCLASS = {52C40 (05B35)},
  MRNUMBER = {4880415},
MRREVIEWER = {Donggyu\ Kim},
}

@article {Ardila2006,
    AUTHOR = {Ardila, Federico and Klivans, Caroline J.},
     TITLE = {The {B}ergman complex of a matroid and phylogenetic trees},
   JOURNAL = {J. Combin. Theory Ser. B},
  FJOURNAL = {Journal of Combinatorial Theory. Series B},
    VOLUME = {96},
      YEAR = {2006},
    NUMBER = {1},
     PAGES = {38--49},
      ISSN = {0095-8956,1096-0902},
   MRCLASS = {05B35},
  MRNUMBER = {2185977},
MRREVIEWER = {Neil\ L.\ White},
       DOI = {10.1016/j.jctb.2005.06.004},
       URL = {https://doi.org/10.1016/j.jctb.2005.06.004},
}

@article {Barmak2011,
    AUTHOR = {Barmak, Jonathan Ariel},
     TITLE = {On {Q}uillen's {T}heorem {A} for posets},
   JOURNAL = {J. Combin. Theory Ser. A},
  FJOURNAL = {Journal of Combinatorial Theory. Series A},
    VOLUME = {118},
      YEAR = {2011},
    NUMBER = {8},
     PAGES = {2445--2453},
      ISSN = {0097-3165,1096-0899},
   MRCLASS = {55P10 (05C10 05E45 06A07)},
  MRNUMBER = {2834186},
MRREVIEWER = {Petar\ Pave\v si\'c},
       DOI = {10.1016/j.jcta.2011.06.008},
       URL = {https://doi.org/10.1016/j.jcta.2011.06.008},
}

@book {Bjorner1999,
    AUTHOR = {Bj\"orner, Anders and Las Vergnas, Michel and Sturmfels, Bernd
              and White, Neil and Ziegler, G\"unter M.},
     TITLE = {Oriented matroids},
    SERIES = {Encyclopedia of Mathematics and its Applications},
    VOLUME = {46},
   EDITION = {Second},
 PUBLISHER = {Cambridge University Press, Cambridge},
      YEAR = {1999},
     PAGES = {xii+548},
      ISBN = {0-521-77750-X},
   MRCLASS = {52B40 (05B35 52C35)},
  MRNUMBER = {1744046},
       DOI = {10.1017/CBO9780511586507},
       URL = {https://doi.org/10.1017/CBO9780511586507},
}

@article {Edelman1985,
    AUTHOR = {Edelman, Paul H. and Walker, James W.},
     TITLE = {The homotopy type of hyperplane posets},
   JOURNAL = {Proc. Amer. Math. Soc.},
  FJOURNAL = {Proceedings of the American Mathematical Society},
    VOLUME = {94},
      YEAR = {1985},
    NUMBER = {2},
     PAGES = {221--225},
      ISSN = {0002-9939,1088-6826},
   MRCLASS = {52A25 (06A10 51M20 55P10 57Q99)},
  MRNUMBER = {784167},
MRREVIEWER = {Udo\ Pachner},
       DOI = {10.2307/2045379},
       URL = {https://doi.org/10.2307/2045379},
}

@article {Hepworth2017,
    AUTHOR = {Hepworth, Richard and Willerton, Simon},
     TITLE = {Categorifying the magnitude of a graph},
   JOURNAL = {Homology Homotopy Appl.},
  FJOURNAL = {Homology, Homotopy and Applications},
    VOLUME = {19},
      YEAR = {2017},
    NUMBER = {2},
     PAGES = {31--60},
      ISSN = {1532-0073,1532-0081},
   MRCLASS = {55N35 (05C10)},
  MRNUMBER = {3683605},
MRREVIEWER = {Haimiao\ Chen},
       DOI = {10.4310/HHA.2017.v19.n2.a3},
       URL = {https://doi.org/10.4310/HHA.2017.v19.n2.a3},
}

@article {Hepworth2022,
    AUTHOR = {Hepworth, Richard},
     TITLE = {Magnitude cohomology},
   JOURNAL = {Math. Z.},
  FJOURNAL = {Mathematische Zeitschrift},
    VOLUME = {301},
      YEAR = {2022},
    NUMBER = {4},
     PAGES = {3617--3640},
      ISSN = {0025-5874,1432-1823},
   MRCLASS = {55N35 (18D20 18F99 51F99)},
  MRNUMBER = {4449723},
MRREVIEWER = {David\ Matthew\ Freeman},
       DOI = {10.1007/s00209-022-03013-8},
       URL = {https://doi.org/10.1007/s00209-022-03013-8},
}

@misc{Koizumi2026,
  author      = {Koizumi, Junnosuke and Liu, Ye},
  title       = {Magnitude homology of real hyperplane arrangements},
  year        = {2026},
  eprint      = {2604.03718},
  eprinttype  = {arxiv},
  eprintclass = {math.CO},
}

@misc{Koizumi2026a,
  author      = {Koizumi, Junnosuke},
  title       = {Magnitude homology of tope graphs},
  year        = {2026},
  eprint      = {2607.11863},
  eprinttype  = {arxiv},
  eprintclass = {math.CO},
}

@article {Leinster2013,
    AUTHOR = {Leinster, Tom},
     TITLE = {The magnitude of metric spaces},
   JOURNAL = {Doc. Math.},
  FJOURNAL = {Documenta Mathematica},
    VOLUME = {18},
      YEAR = {2013},
     PAGES = {857--905},
      ISSN = {1431-0635,1431-0643},
   MRCLASS = {51F99 (18D20 28A75 49Q20 52A38 53C65)},
  MRNUMBER = {3084566},
MRREVIEWER = {J.\ B\"ohm},
}

@article {Leinster2019,
    AUTHOR = {Leinster, Tom},
     TITLE = {The magnitude of a graph},
   JOURNAL = {Math. Proc. Cambridge Philos. Soc.},
  FJOURNAL = {Mathematical Proceedings of the Cambridge Philosophical
              Society},
    VOLUME = {166},
      YEAR = {2019},
    NUMBER = {2},
     PAGES = {247--264},
      ISSN = {0305-0041,1469-8064},
   MRCLASS = {05C31 (05C50 18D20 57M15)},
  MRNUMBER = {3903118},
MRREVIEWER = {Lorenzo\ Traldi},
       DOI = {10.1017/S0305004117000810},
       URL = {https://doi.org/10.1017/S0305004117000810},
}

@book {Orlik1992,
    AUTHOR = {Orlik, Peter and Terao, Hiroaki},
     TITLE = {Arrangements of hyperplanes},
    SERIES = {Grundlehren der mathematischen Wissenschaften [Fundamental
              Principles of Mathematical Sciences]},
    VOLUME = {300},
 PUBLISHER = {Springer-Verlag, Berlin},
      YEAR = {1992},
     PAGES = {xviii+325},
      ISBN = {3-540-55259-6},
   MRCLASS = {52B30 (14F35 20F36 20F55 32S25 57N65)},
  MRNUMBER = {1217488},
MRREVIEWER = {Michel\ Yves\ Jambu},
       DOI = {10.1007/978-3-662-02772-1},
       URL = {https://doi.org/10.1007/978-3-662-02772-1},
}

@book {Oxley2011,
    AUTHOR = {Oxley, James},
     TITLE = {Matroid theory},
    SERIES = {Oxford Graduate Texts in Mathematics},
    VOLUME = {21},
   EDITION = {Second},
 PUBLISHER = {Oxford University Press, Oxford},
      YEAR = {2011},
     PAGES = {xiv+684},
      ISBN = {978-0-19-960339-8},
   MRCLASS = {05-01 (05B35 90C27)},
  MRNUMBER = {2849819},
MRREVIEWER = {Maruti\ M.\ Shikare},
       DOI = {10.1093/acprof:oso/9780198566946.001.0001},
       URL = {https://doi.org/10.1093/acprof:oso/9780198566946.001.0001},
}

@article {Rau2022,
    AUTHOR = {Rau, Johannes and Renaudineau, Arthur and Shaw, Kris},
     TITLE = {Real phase structures on matroid fans and matroid
              orientations},
   JOURNAL = {J. Lond. Math. Soc. (2)},
  FJOURNAL = {Journal of the London Mathematical Society. Second Series},
    VOLUME = {106},
      YEAR = {2022},
    NUMBER = {4},
     PAGES = {3687--3710},
      ISSN = {0024-6107,1469-7750},
   MRCLASS = {52C40 (05B35 14T15)},
  MRNUMBER = {4524208},
MRREVIEWER = {Galen\ Dorpalen-Barry},
       DOI = {10.1112/jlms.12671},
       URL = {https://doi.org/10.1112/jlms.12671},
}

@book {Schrijver2003,
    AUTHOR = {Schrijver, Alexander},
     TITLE = {Combinatorial optimization. {P}olyhedra and efficiency. {V}ol.
              {B}},
    SERIES = {Algorithms and Combinatorics},
    VOLUME = {24,B},
      NOTE = {Matroids, trees, stable sets,
              Chapters 39--69},
 PUBLISHER = {Springer-Verlag, Berlin},
      YEAR = {2003},
     PAGES = {i--xxxiv and 649--1217},
      ISBN = {3-540-44389-4},
   MRCLASS = {90-02 (05-02 52B55 68Q25 68R10 90C27 90C35 90C57)},
  MRNUMBER = {1956925},
MRREVIEWER = {Alexander\ I.\ Barvinok},
}

@misc{Shaw2026,
  author      = {Shaw, Kris and Yuen, Chi Ho},
  title       = {Filtrations of tope spaces of oriented matroids},
  year        = {2026},
  eprint      = {2501.11295},
  eprinttype  = {arxiv},
  eprintclass = {math.CO},
  note        = {Version 2},
}

@incollection {Stanley2007,
    AUTHOR = {Stanley, Richard P.},
     TITLE = {An introduction to hyperplane arrangements},
 BOOKTITLE = {Geometric combinatorics},
    SERIES = {IAS/Park City Math. Ser.},
    VOLUME = {13},
     PAGES = {389--496},
 PUBLISHER = {Amer. Math. Soc., Providence, RI},
      YEAR = {2007},
      ISBN = {978-0-8218-3736-8; 0-8218-3736-2},
   MRCLASS = {52C35 (05B35 55R80)},
  MRNUMBER = {2383131},
       DOI = {10.1090/pcms/013/08},
       URL = {https://doi.org/10.1090/pcms/013/08},
}

@article {Yagi2026,
    AUTHOR = {Yagi, Yukino and Yoshinaga, Masahiko},
     TITLE = {Reconstruction of oriented matroids from {V}archenko-{G}elfand
              algebras},
   JOURNAL = {Int. Math. Res. Not. IMRN},
  FJOURNAL = {International Mathematics Research Notices. IMRN},
      YEAR = {2026},
    NUMBER = {11},
     PAGES = {Paper No. rnag109, 25},
      ISSN = {1073-7928,1687-0247},
   MRCLASS = {52C40 (05B35 52C35)},
  MRNUMBER = {5082126},
       DOI = {10.1093/imrn/rnag109},
       URL = {https://doi.org/10.1093/imrn/rnag109},
}

\end{document}